%% file: main.tex
\documentclass[11pt]{article}
\usepackage{arxiv}

\usepackage[T1]{fontenc}        
\usepackage[utf8]{inputenc}     

\usepackage[numbers]{natbib}
\usepackage{authblk} 

\usepackage{amsmath}            
\usepackage{amsthm}             
\usepackage{amsfonts}           
\usepackage{mathtools}          
\mathtoolsset{showonlyrefs}

\usepackage{bm}                 
\usepackage{mathrsfs}           
\usepackage{upgreek}            
 
\usepackage{amssymb}            

\usepackage{caption}            
\usepackage{subcaption}         
\usepackage{array}              
\usepackage{multirow}           
\usepackage{tabularx}           

\usepackage{nomencl}            
\usepackage{enumerate}          

\RequirePackage[numbers]{natbib}
\RequirePackage[colorlinks,citecolor=blue,urlcolor=blue]{hyperref}
\RequirePackage{graphicx}

\theoremstyle{plain}

\newtheorem{theorem}{Theorem}[section]
\newtheorem{lemma}[theorem]{Lemma}
\newtheorem{corollary}[theorem]{Corollary}
\newtheorem{statement}[theorem]{Statement}
\newtheorem{proposition}[theorem]{Proposition}

\theoremstyle{definition}

\newtheorem{remark}[theorem]{Remark}

\theoremstyle{remark}

\input{mydef.tex}

\title{Multiplier bootstrap for Bures--Wasserstein barycenters}

\author[1]{Alexey Kroshnin\thanks{\texttt{alex.kroshnin@gmail.com}}}
\author[1]{Vladimir Spokoiny}
\author[1]{Alexandra Suvorikova}

\affil[1]{Weierstrass Institute for Applied Analysis and Stochastics}

\date{}

\begin{document}
\maketitle
   
\begin{abstract}
    This study proposes a bootstrap-based method for uncertainty quantification in two important statistical scenarios. First, we approximate the sampling distribution of empirical barycenters under the Bures--Wasserstein metric using a reweighted estimator. Our theoretical results guarantee the accuracy of this approximation and enable the construction of data-driven confidence sets. The methodology is validated through experiments on graph-structured data, including stochastic block models and brain connectomes. Additionally, we compare bootstrap-based confidence sets with the asymptotic confidence sets obtained in \cite{kroshnin2019statistical}, evaluating both their statistical performance and computational complexity.
    
    Second, we investigate the generalized bootstrap framework for $M$-estimators without requiring a specific resampling scheme, thus covering both weighted and resampling methods under mild conditions.
    
    Both contributions rely on a novel Gaussian approximation result for $M$-estimators.
\end{abstract}
    



\section{Introduction}
\label{sec:intro}

A Fr{\'e}chet mean, also known as a barycenter, provides a natural generalization of the average---or centre of mass---to spaces endowed with a particular notion of distance. 
Formally, for a set of points $S_1, \dots, S_n$ lying in a metric space $(\mathcal{S}, d)$, their barycenter is a minimizer of the variance:
\[
B \in \aargmin_{Q\in \mathcal{S}} \sumi d^{2}(Q, S_i).
\]
Replacing the sample $S_1, \dots, S_n$ by a probability measure $\mu$ on $\mathcal{S}$ yields a definition of a \emph{population barycenter}:
\[
B \in \aargmin_{Q\in \mathcal{S}} \int_{\mathcal{S}} d^{2}(Q, S) d \mu(S).
\]

In this work, we consider the space $\mathbf{H}_{++}(d)$ of positive semi-definite (PSD) matrices that is common in many practical scenarios \cite{huang2025towards, maunu2023bures, xu2025wasserstein}. In numerous applications---such as brain-computer interface research or network analysis---computing a barycenter of a set of PSD enhances the interpretability and stability of subsequent analyses \cite{zheng2023barycenter, haasler2024bures}. 
A key ingredient in defining a barycenter is the choice of distance. 
In this work, we focus on the Bures--Wasserstein distance, introduced in \cite{bhatia2018bures}, which is particularly useful for analysis of graph-structured data \citep{haasler2024bures, maretic2022wasserstein, maretic2022fgot} and EEG signal studies \cite{zalles2024an}.


The Bures--Wasserstein distance on the space of $d\times d$ positive semi-definite Hermitian matrices $\H_{+}(d)$ is defined as  
\[
\bw^2(Q, S) \eqdef
\tr(Q) + \tr(S) - 2 \tr\Bigl(\bigl(S^{1/2} Q S^{1/2}\bigr)^{1/2}\Bigr),
\quad
Q,\; S \in \H_{+}(d).
\]
The dual name of the Bures--Wasserstein distance itself indicates a connection with both quantum information theory \cite{bures1969extension} and optimal transport (OT) theory \cite{takatsu2011wasserstein}. In the OT framework, one defines a transportation cost and then seeks the most efficient strategy to transform one probability measure into another. This yields a geometrically meaningful metric that quantifies the distance between probability measures \citep{ambrosio_gradient_2008, villani_optimal_2009, santambrogio2015optimal}. The Bures--Wasserstein distance is closely related to the $2$-Wasserstein distance, one of the most widely used OT distances. Specifically, the $2$-Wasserstein distance between two centred Gaussian measures coincides with the Bures--Wasserstein distance between their covariance matrices \citep{takatsu2011wasserstein, bhatia2018bures}. The $2$-Wasserstein barycenters provide a geometrically meaningful method to average distributions. In particular, the barycenter of Gaussian distributions is itself Gaussian with the covariance matrix equal to the Bures--Wasserstein barycenter of their covariance matrices \cite{agueh2011barycenters}.

We refer to the Fr{\'e}chet mean with respect to the distance $\bw$ as the Bures--Wasserstein barycenter. Let $\H_{++}(d)$ be the space of positive-definite $d\times d$ Hermitian matrices and let $\mathcal{M}_2\bigl(\H_{++}(d)\bigr)$ be the space of non-zero finite Borel measures on $\H_{++}(d)$ with a finite second moment. We define the barycenter mapping from $\mathcal{M}_2\bigl(\H_{++}(d)\bigr)$ to $\H_{++}(d)$ as
\begin{equation}
\label{def:bar_map}
   \mathscr{B} \colon \mu \mapsto B_{\mu} \eqdef \argmin_{Q \in \H_{++}(d)} \int\limits_{\H_{++}(d)} \bw^2(Q, S) \dd \mu(S), 
\end{equation}
Note that $B_\mu$ is unique by Theorem~2.1 in~\citep{kroshnin2019statistical}, ensuring that the mapping $\mathscr{B}$ is well-defined.

In many scenarios, the observed data is random. Consequently, the barycenters computed from such samples are data-driven estimators for the unknown population barycenter. This naturally raises the question of how stable the estimators are and how they fluctuate.

Some classical results apply because the barycenters are $M$-estimators \citep{van2006empirical}. Numerous studies have addressed explicit convergence rates, concentration inequalities, and large deviation results for empirical Wasserstein barycenters \citep{ahidar2020convergence, brunel2024concentration, le2022fast, jaffe2024large}. Some works establish the Central Limit Theorem in various settings \citep{kroshnin2019statistical, carlier2021entropic}.  Limit theorems admit constructing confidence sets. However, in practice, this approach has various limitations, such as small sample size, high computational cost, etc. A fundamentally different mechanism for constructing confidence sets is based on the bootstrap approach. 

Bootstrapping techniques have attracted much attention due to their algorithmic simplicity and computational tractability since their introduction by \citet{efron1979bootstrap}. \citet{spokoiny_zhilova_2015} apply multiplier bootstrap to construct likelihood-based confidence sets.
\citet{chen2020robust} investigate the case of heavy-tailed data. 
\citet{naumov2019bootstrap} validate bootstrap approximation for spectral projectors in the case of Gaussian data.
\citet{cheng2010bootstrap} provides approximation rates for multiplier bootstrap for $M$-estimators in semi-parametric models.
\citet{lee2020bootstrap} propose a resampling procedure for $M$-estimators for non-standard cases.
For more examples, we recommend a survey by~\citet{mammen2012bootstrap}.

The primary goal of this study is to quantify the uncertainty associated with the empirical barycenter. To this end, we introduce two related bootstrap frameworks: a multiplier bootstrap tailored to Bures--Wasserstein barycenters, and the generalized bootstrap for $M$-estimators, and theoretically validate these methods.

\subsection{Contribution of this paper}

The first result uses the reweighted empirical barycenter $B_w$ to approximate the distribution of the Bures--Wasserstein distance $\bw(B,B_n)$, where $B$ is the unknown population barycenter and $B_n$ is its empirical counterpart. Theorem~\ref{thm:bootstrap_bw} shows that under mild assumptions covering a wide spectrum of applications
\begin{equation}
\label{eq:approx_main}
    \sqrt{n}\,\bw(B,B_n) \overset{\mathscr{d}}{\approx} \sqrt{n}\,\bw(B_n,B_w),
\end{equation}
with an approximation error of order $n^{-1/2}$ up to a logarithmic factor. In particular, Corollary~\ref{corr:conf_sets} presents theoretical bounds for data-driven confidence sets for $B_n$. 

The experiments illustrate the practical applicability of the multiplier bootstrap for uncertainty quantification associated with $B_n$. We consider graph-structured data, including a weighted stochastic block model and human brain connectomes—complex networks encoding inter-regional connectivity whose analysis is crucial for understanding brain function and behavior \citep{bullmore2009complex,fornito2016fundamentals}. Section~\ref{sec:comp_eff} shows that in the high-dimensional regime the approximation \eqref{eq:approx_main} is computationally more efficient than the asymptotic approximation proposed in \cite{kroshnin2019statistical}, 
which makes it a more practical choice for applications requiring computationally scalable inference methods, especially in high dimensions.

The second result extends the generalized bootstrap framework of \cite{van1996weak} to $M$-estimators. Theorem~\ref{thm:bootstrap_M} validates the bootstrap approximation of the distribution
$\norm{\theta_* - \theta_\mu},$
where $\theta_*$ is the true population $M$-estimator and $\theta_\mu$ its empirical counterpart. To our knowledge, this is the first result that avoids specifying a particular resampling scheme, relying only on broad conditions on the bootstrapped risk. Of note, this approach covers both weighted and resampling bootstrap techniques.

Theorem~\ref{thm:bootstrap_bw} and Theorem~\ref{thm:bootstrap_M} follow from the Gaussian Approximation Result (GAR) derived for $M$-estimators (Proposition~\ref{prop:gar}). We specify GAR for the barycenters and quantify the convergence in terms of a Kolmogorov-type distance, what extends the CLT result obtained in \cite{kroshnin2019statistical}. We note that under additional regularity assumptions one can use GAR result for Wasserstein distance.

\subsection{Organization of the paper and accepted notations}

The paper is organized as follows. Section~\ref{sec:sub_gauss_intro} introduces the multiplier bootstrap in the Bures--Wasserstein space.
Section~\ref{seq:intro_gen_boot} presents the generalized bootstrap for $M$-estimators. Section~\ref{sec:exp} evaluates the performance of the multiplier boostrap method on both synthetic and real data sets. In addition, we compare the approximations constructed using the multiplier bootstrap with those derived from the asymptotic results of \citet{kroshnin2019statistical} and analyze the computational complexities of both methods. 
Finally, the appendices contain proofs.

Throughout this work, we use the following notation. 
Let $\mathbb{H}(d)$ be the set of all $d \times d$ Hermitian matrices, 
and let $\mathbb{H}_{+}(d)$ and $\mathbb{H}_{++}(d)$ denote the sets of 
positive semi-definite and positive definite Hermitian matrices, respectively.
We use $X$ to denote matrices or vectors, while $\o{X}$
denotes an operator. For any matrix $X$, let $\lambda_{\max}(X)$ and 
$\lambda_{\min}(X)$ be its largest and smallest eigenvalues, respectively. 
We use $\norm{X}$, $\norm{X}_{\mathrm{F}}$, $\norm{X}_{1}$, and $\norm{X}_{\psi_{\alpha}}$ to denote the operator norm, Frobenius norm, 1-Schatten norm, and $\psi_{\alpha}$-Orlicz norm, respectively.
Recall, that the $\alpha$-Orlicz ($\alpha \ge 1$) of a real-valued random variable $X$ is
\[
\norm{X}_{\psi_{\alpha}} \eqdef \inf\left\{t\ge 0:\,\E\psi_{\alpha}\left(\frac{|X|}{t} \right) \le 1\right\},
\quad\text{where}~~
\psi_{\alpha}(x) \eqdef e^{x^{\alpha}} - 1.
\]
If $\norm{X}_{\psi_{1}} < +\infty$, $X$ is sub-exponential, and if $\norm{X}_{\psi_{2}} <  +\infty$, $X$ is sub-Gaussian. The inner product associated with the Frobenius norm is given by $\langle X, Y \rangle$. We define the condition number of a matrix or an operator $X$ as $\kappa (X) \eqdef \norm{X}\: \norm{X^{-1}}$. 
The symbol $\otimes$ denotes the tensor product. Throughout the text, we set $\log(x) \eqdef \max\{1, \ln(x)\}.$ Furthermore, for matrices (or operators) $X$ and $A$, let $r(X, A) \eqdef \norm*{X^{-1/2} A X^{-1/2} - I}$. Finally, $\CONST$ denotes a generic absolute constant whose value may change from instance to instance, and the notation $\overset{\mathrm{d}}{\approx}$ indicates closeness in distribution.

\section{Multiplier bootstrap for Bures--Wasserstein barycenters}
\label{sec:sub_gauss_intro}

Before delving into the multiplier‑bootstrap procedure, we first state the key conditions for our theoretical analysis. Specifically, we assume that the data distribution $P$ and the multiplier (or weight) distribution $W$ are sub-exponential. 
Specifically, we assume that there exists a constant $v_P$ such that
\begin{equation}
\label{asm:subexp}
    \norm*{\tr S}_{\psi_1} \le v_{P}, \quad S\sim P.
    \tag{$P$}
\end{equation} 
The corresponding population barycenter is $B = \mathscr{B}(P)$ with the mapping $\mathscr{B}$ introduced by \eqref{def:bar_map}. Let $S_1, \dots, S_n \in \H_{++}(d)$ be an i.i.d.\ sample from a distribution $P$. The corresponding empirical barycenter is
\begin{equation}
\label{def:B_n}
    B_n \eqdef \mathscr{B}(P_n), 
    \quad P_n \eqdef \frac{1}{n}\sum^n_{i=1}\delta_{S_i}, 
\end{equation}
where $\delta_{S}$ is the Dirac measure at $S$.
To approximate the fluctuation of $B_n$ in the vicinity of the population barycenter $B$, we use the multiplier bootstrap. Given a set of  weights, we construct a reweighted empirical barycenter
$B_w \in \H(d)_{+}$ which admits the following non-asymptotic approximation (as we will show later):
\begin{equation}
\label{eq:goal}
   \sqrt{n} \bw(B, B_n) \overset{\mathrm{d}}{\approx} \sqrt{n}\bw(B_n, B_w). 
\end{equation}
Specifically, let $w_1, \dots, w_n$ be real-valued i.i.d.\ non-negative weights sampled from a sub-exponential distribution $W$, i.e., there exists a constant $v_w$, such that
\begin{equation}
\label{assumption:subgauss}
    \norm*{w - 1}_{\psi_1} \le v_w, 
    \quad
    \E w = \var w = 1, 
    \quad w\sim W.
    \tag{$W$}
\end{equation}
Some specific examples of $W$ are the exponential, the Poisson, or the Bernoulli distributions. 
We set
\[
B_w = \mathscr{B}(P_w), \quad P_w = \frac{1}{n} \sum^n_{i=1} w_i \delta_{S_i}.
\]
Notably, if the weight distribution is Bernoulli or Poisson, it is possible for $\sumi w_i = 0$. In this case the weighted empirical barycenter is degenerated, i.e., $B_w = 0$.

To establish the validity of the approximation \eqref{eq:goal} and to derive non-asymptotic bounds, we rely on the properties of OT maps between positive definite Hermitian matrices. 
An OT map transforms one distribution to another with optimal cost, and in the case of Gaussians it is known to be linear \cite{takatsu2011wasserstein}.
It can be generalized to the Bures--Wasserstein space \cite{bhatia2018bures}, providing a linear transformation of an Hermitian operator related to the Bures--Wasserstein distance.
Specifically, for any $Q, S \in \H_{++}(d)$, we denote
\[
T^{S}_{Q} \eqdef Q^{-1/2} \left(Q^{1/2} S Q^{1/2} \right)^{1/2} Q^{-1/2} \in \H_{++}(d),
\]
so that $T^{S}_{Q} Q T^{S}_{Q} = S$ and $\bw(S, Q) = \tr (T^{S}_{Q} - I) Q (T^{S}_{Q} - I)$.
Furthermore, $T^{S}_{Q}$ is differentiable in the Fr{\'e}chet sense (see Lemma~A.2 by \cite{kroshnin2019statistical}):
\begin{equation}
\label{def:dTQS}
    T^{S}_{Q + X} = T^{S}_{Q} + \dT{Q}{S}(X) + o(\norm{X})
    \quad
    \text{as $\norm{X} \to 0$,} \quad X \in \H(d),
\end{equation}
where $\dT{Q}{Q} \colon \H(d) \to \H(d)$ is a negative semi-definite operator. Appendix~\ref{sec:bw_geometry} presents its explicit form and discusses properties. 

Now we are ready to present the main result of this section.

\begin{theorem}
\label{thm:bootstrap_bw}
    Let Assumptions~\eqref{asm:subexp} and~\eqref{assumption:subgauss} be fulfilled. 
    Let $p_0$ be the probability of observing $w_i = 0$, i.e., $p_0 = \P_{w}\{w_i = 0\}$. 
    Denote
    \begin{gather*}
        \o{\Sigma} \eqdef \E \left(T^{S_1}_{B} - I \right) \otimes \left(T^{S_1}_{B} - I \right),
        \quad \sigma^2_T \eqdef \E \left\|T^{S_1}_B-I \right\|^2_{\F},\\
        \A \eqdef \left(- \frac{1}{2} \dT{B}{B}\right)^{1/2}, \quad 
        \o{F} \eqdef -\E\dT{B}{S_1}, 
        \quad \sigma^2_F \eqdef  \norm*{\E \left(\dT{B}{S_1} - \o{F} \right)^2}.
    \end{gather*}
    Then, there exists a function $\Gamma(n, \t) \ge 0$ such that, with probability at least $1 - \CONST e^{-\t}$, it holds
    \[
    \sup_{z\ge 0} \abs*{
    \P\left\{\sqrt{n}\bw(B_n,B) \le z\right\}
	- \P\left\{ \sqrt{n}\bw(B_w,B_n) \le z ~|~\mu\right\}} \leq \Gamma(n, \t) + p^n_0.
    \]
    Moreover, for sufficiently large $n$ (depending on $\t$), 
    \begin{align*}
        \Gamma(n, \t) \lesssim {}& d^3\sqrt{\frac{C_G}{n}} + \kappa(\A\o{F}^{-1} \o{\Sigma}\o{F}^{-1}  \A)\norm*{\A}^2\norm*{\o{F}^{-1}}^2\sigma^2_T \times \\
        & \times \left(\sqrt{\frac{\hat{C}_\eps}{n}\left( \t + \log \frac{nd}{\hat{C}_\eps}\right)} + \sqrt{\frac{C_T}{n}\left(\t+d^2\right)} \right),
    \end{align*}
    
    where $\kappa(\cdot)$ is the condition number. The constants $\hat{C}_\eps, C_T, C_G > 0$ depend on the distribution.
\end{theorem}

The explicit expressions for the constants $\hat{C}_\eps, C_T, C_G$ are given by~\eqref{def:ce_hat}, \eqref{def:ct}, and~\eqref{def:cg_hat}, respectively. The explicit condition on the sample size $n$ is given by~\eqref{def:size_n}. The proof is in Appendix~\ref{sec:proof_sub_gauss}.

\begin{remark}
    The rate $\frac{d^3}{\sqrt{n}}$ is due to the technique used in the proof of the Gaussian approximation results (see Lemma~\ref{lemma:GAR_subg} and~\ref{lemma:gar_sub_boot}). 
    Specifically, we get $d^3$ instead of $d^{3/2}$ because the results are in the space of $d\times d$ matrices.
\end{remark}

The next trivial corollary guarantees the validity of the bootstrap-based procedure of constructing confidence sets.

\begin{corollary}
\label{corr:conf_sets}
    Let the assumptions of Theorem~\ref{thm:bootstrap_bw} be fulfilled. Then with probability at least $1 - \CONST e^{-\t}$ we have
    \[
    \abs*{\P\left(\sqrt{n} \bw(B, B_n) \le \hat{z}_{\alpha}\right) - \alpha} \le \Gamma(n, \t) + p^n_0,
    \]
    where 
    $
    \hat{z}_{\alpha} \eqdef \inf\left\{z \ge 0:\,\P\left(\sqrt{n}\bw(B_w, B_n) \le z \right) \ge \alpha \right\}.
    $
\end{corollary}

Largely, the validity of Theorem~\ref{thm:bootstrap_bw} follows from the general framework for $M$-estimators presented in Section~\ref{seq:intro_gen_boot}. However, there is a part of the proof specific to the Bures--Wasserstein distance, which we will discuss below. This part relies on the linearization of the Bures--Wasserstein. The OT maps $T^{S_1}_B, \dots, T^{S_n}_B$ play a key role in this linearization. Recall that by construction $\E T^{S_1}_B = I$ and denote 
\[
T_n \eqdef \frac{1}{n} \sumi T^{S_i}_B - I, 
\quad
T_w \eqdef \frac{1}{n} \sumi w_i\left(T^{S_i}_B - I\right). 
\]
Lemma~\ref{lemma:Qw_expansion} claims that for sufficiently large $n$ it holds with high probability that
\begin{align}
\label{eq:lin1}
    &\abs*{\bw(B, B_n) - \norm*{\A \o{F}^{-1} T_n}_{\F}}
	\lesssim \frac{1}{\sqrt{n}} \norm*{\A \o{F}^{-1} T_n}_{\F},\\
    \label{eq:lin2}
    &\abs[\big]{\bw(B_{w}, B_{n}) - \norm{\A \o{F}^{-1} (T_{w} - T_n)}_{\F}} \lesssim \frac{1}{\sqrt{n}} \norm{\A\o{F}^{-1} \left(T_{w} - T_{n} \right) }_{\F} + \frac{1}{\sqrt{n}}\norm{\A}
    \norm{\o{F}^{-1} T_{n}}_{\F},
\end{align}
where $\norm{\cdot}_\F$ denotes Frobenius norm. The deviation bounds on $\norm{\A\o{F}^{-1}T_n}_\F$ and $\norm{\A\o{F}^{-1}(T_w-T_n)}_{\F}$ follow from Assumptions~\eqref{assumption:subgauss} and~\eqref{asm:subexp}, because sub-exponential tail behaviour of $\tr S_i$ implies sub-Gaussian tail behaviour of $T_n$ and $T_w$. The next lemma justifies this fact (the proof is in Appendix~\ref{sec:proof_sub_gauss}).
\begin{lemma}
\label{lemma:aux_behaviour}
    Let Assumption~\eqref{asm:subexp} hold. 
    Then for a fixed $Q \in \H_{++}(d)$ and $S \sim P$
    there exist some constants $v_S, v_T, v_F > 0$, such that 
    $\norm*{\norm{S}^{1/2}}_{\psi_2} \le v_{S}$, $\norm*{\norm[\big]{T^{S}_{Q}}_{\F}}_{\psi_2} \le v_{T}$, 
    $\norm*{\norm[\big]{\dT{Q}{S}}}_{\psi_2} \le v_F$.
\end{lemma}

\section{Generalized bootstrap for \texorpdfstring{$M$}{M}-estimators}
\label{seq:intro_gen_boot}

The multiplier bootstrap for barycenters is a special case of a more general result, generalized bootstrap for $M$-estimators \citep{van1996weak}.
$M$-estimators form a broad class of statistical estimators defined through optimization of a given risk induced by a loss function $\ell$. This formulation provides a unifying framework for a wide range of estimation procedures, including maximum likelihood, least squares, etc. 
Let $\mathbb{X}$ be a measurable space, $H$ be a separable Hilbert space, and $\Theta \subset H$ be a convex closed parameter set. We consider a loss function $\ell \colon \Theta \times \mathbb{X} \rightarrow \R$, measurable in $x$ and l.s.c.\ in $\theta$. The associated $M$-estimator is a (partial) map
\[
\mathcal{M} \colon \mu \mapsto \theta_{\mu} \eqdef \argmin_{\theta \in \Theta} \int\limits_{\mathbb{X}} \ell(\theta, x) \dd \mu(x),  
\]
where $\mu$ is a probability measure on $\mathbb{X}$. Note that
$\theta_\mu$ is, a priori, not unique, so we fix any measurable choice of $\theta_\mu$ (which exists due to the lower semicontinuity of $\ell$).
In practice, $\mu$ is random---for example, a data-driven empirical distribution---so $\theta_\mu$ acts as a plug-in estimator of the unknown true parameter $\theta_* = \mathcal{M}(\mu_*)$. 
The generalized bootstrap provides data-driven bounds controlling the fluctuation of $\theta_\mu$. That is, for a properly selected $\hat{\mu}$ depending on $\mu$, one constructs $\theta_{\hat{\mu}} = \mathcal{M}(\hat{\mu})$, such that there exists a function $\Gamma(\t) > 0$, that with probability $1 - \CONST e^{-\t}$
\begin{equation}
\label{eq:K_distance}
    \sup_{z}\,\abs*{\P\left\{ d(\theta_\mu, \theta_*) \le z\right\} - \P\left\{ d(\theta_{\hat{\mu}}, \theta_\mu) \le z \big| \mu \right\} } \le \Gamma(\t),
\end{equation}
where $d$ is some appropriate distance on $\Theta$.
To show the validity of this bound, it is enough to use two tools---the Taylor decomposition and the Gaussian approximation---and to impose some mild assumptions on the statistical behaviour of the gradient of the loss function $\ell$ and its Hessian. 
In what follows, we assume the $\nabla_{\theta} \ell$ and the Hessian $\nabla^2_{\theta} \ell$ exist (in the Fr{\'e}chet sense) and are measurable. Furthermore, we require that they are integrable with respect to $\mu_*$ and a.s.\ integrable with respect to $\mu$ and $\hat{\mu}$.
We slightly extend the definition of $\mathcal{M}$ by allowing $\mu$ to be any non-negative finite measure on $\mathbb{X}$, not necessarily of unit mass.
Let $\mathcal{M}_{a} \subset \mathcal{M}_{+}(\mathbb{X})$ be the set of measures for which $\theta_\mu$ exists and the integrability assumptions hold.
Fix $\nu \in \mathcal{M}_{a}$. For any $\theta \in \Theta$ we write $\theta_t \eqdef (1-t) \theta_* + t \theta$ and define
\begin{gather}
    \o{D}_{\nu}(\theta) \eqdef \int\limits_{\mathbb{X}}\int\limits_0^1 \nabla_{\theta}^2 \ell(\theta_t, x) \dd t \dd \nu(x),
    \quad 
    \o{F}_\nu(\theta) \eqdef \int\limits_{\mathbb{X}}\nabla^2_{\theta}\ell(\theta, x) \dd \nu(x), \\ 
    g_{\nu}(\theta) \eqdef \int\limits_{\mathbb{X}}\nabla_{\theta}\ell(\theta, x)\dd \nu(x).
\end{gather}
We denote
\begin{equation*}
    \o{F} \eqdef \o{F}_{\mu_*}(\theta_*), \quad
    g_\nu \eqdef g_\nu(\theta_*), \quad \o{D}_\nu \eqdef \o{D}_\nu(\theta_\nu),
    ~~\text{where}~~\theta_\nu \eqdef \mathcal{M}(\nu).
\end{equation*}
We also assume that $\theta_{\nu}$ satisfies the first-order optimality condition:
\begin{equation}
    g_{\nu}(\theta_\nu) = 0, \quad \nu \in \mathcal{M}_{a} .
\end{equation}

The next lemma is the key ingredient for the Gaussian Approximation Result (GAR). It relies on Taylor's decomposition. 
Let $\o{I}$ be the identity operator and denote for any $\o{X} \in \H_{++}(d)$, $\o{Y} \in \H(d)$,
\[
r(\o{X}, \o{Y}) \eqdef \norm{\o{X}^{-1/2} \o{Y} \o{X}^{-1/2} - \o{I} },
\quad
r_\nu \eqdef r(\o{F}, \o{D}_\nu)~~\text{for any}~~ \nu \in \mathcal{M}_{a}.
\]

\begin{lemma}
\label{lemma:aux_g_nu}
    Let $\mu$ and $\nu$ in $\mathcal{M}_{a}$ be such that $r_\mu < 1$, $r_\nu < 1$. 
    Then for any bounded linear operator $\A$ it holds that
    \begin{align*}
        & \norm*{\A (\theta_\mu - \theta_\nu) + \A \o{F}^{-1} \left( g_\mu - g_\nu \right)}  
      \\
      &\le \norm{\A \o{F}^{-1/2}} \left[\frac{r_\mu}{1 - r_\mu} \norm{\o{F}^{-1/2} \left(g_{\mu} - g_{\nu}\right)} + \left(\frac{r_\mu}{1 - r_\mu} + \frac{r_\nu}{1 - r_\nu} \right) \norm{\o{F}^{-1/2} g_\nu}\right]
    \end{align*}
\end{lemma}

\begin{proof}
    At the optimal $\theta_{\mu}$ we have $g_{\mu}(\theta_{\mu}) = 0$, so Taylor's decomposition is written as
    \begin{equation*}
    \label{eq:aux_Taylor}
        g_{\mu}(\theta_\mu) - g_{\mu} = \o{D}_\mu[\theta_\mu - \theta_*], ~~\text{which entails}~~\theta_\mu - \theta_* = - \o{D}^{-1}_\mu g_\mu
    \end{equation*}
    (note that $\o{D}_\mu$ is invertible since $r_\mu < 1$).
    The same holds for $\theta_\nu$.
    This ensures
    \[
    \theta_{\mu} - \theta_{\nu} = - \o{D}^{-1}_{\mu} g_{\mu} + \o{D}^{-1}_{\nu} g_{\nu} 
    = - \o{D}^{-1}_{\mu} \left[g_{\mu} - g_{\nu} \right] + \left[\o{D}^{-1}_{\nu} - \o{D}^{-1}_{\mu} \right] g_{\nu}.
    \]
    By definition of $r(\cdot, \cdot)$ and conditions of the lemma, it holds
    $r(\o{F}^{-1}, \o{D}^{-1}_\mu) \le \frac{r_\mu}{1-r_\mu}$. The same holds for $\nu$. The claim follows.
\end{proof}

Now we are to derive GAR result. Specifically, we will show that for a properly chosen centred Gaussian vectors $G$ and $\hat{G}$ it holds $\norm{\o{A}(\theta_* - \theta_\mu)} \overset{\mathscr{d}}{\approx}\norm{\o{AF}^{-1}G}$ and $\norm{\o{A}(\theta_{\hat{\mu}} - \theta_\mu)} \overset{\mathscr{d}}{\approx}\norm{\o{AF}^{-1}\hat{G} }$. To ensure these approximations, we introduce several assumptions on the statistical behaviour of $\nabla_\theta \ell$ and $\nabla^2_{\theta} \ell$. First, we assume that there exists a Borel set $\mathscr{A}_{t} \subset \mathcal{M}(\mathbb{X})$, such that $\P\{\mu \in \mathscr{A}_{t} \} \ge 1 - \CONST e^{-t}$.
All assumptions on the bootstrap measure hold on the event $\{\mu \in \mathscr{A}_{t}\}$. We assume that there exist functions $\eps_F(\x)>0$ and $\hat{\eps}_F(\x, \t)>0$, such that 
\begin{align}
\label{asm:r}
    &\P \left\{ r_\mu \le \eps_F(\x) \right\} \ge 1 - \CONST e^{-\x}, ~~~~\text{and}~~~~ 
    \P \left\{
    r_{\hat{\mu}}
    \le \hat{\eps}_F(\x, \t) \bigm| \mu \right\} \ge 1 - \CONST e^{-\x},
\end{align}
with $\mu \in \mathscr{A}_{t}$.
Next, we assume that there exist centred Gaussian vectors $G\sim \ND(0, \o{\Sigma})$ and $\hat{G} \sim \ND(0, 
\hat{\o{\Sigma}})$ (with $\hat{\o{\Sigma}}$ depending on $\mu$), such that for some $\eps_G > 0$ and for some $\hat{\eps}_G(\t) > 0$ it holds
\begin{align}
\label{asm:gar_1}
    &\sup_{C\in \mathscr{C}}\,\abs*{\P\left\{g_{\mu} \in C \right\} -  \P\left\{ G \in C \right\}} \le \eps_G\\
    \label{asm:gar_2}
    &\sup_{C\in \mathscr{C}}\,\abs*{\P\left\{g_{\hat{\mu}} - g_{\mu}\in C \big| \mu \right\} -  \P\left\{ \hat{G} \in C \Bigm| \mu \right\}} \le \hat{\eps}_G(\t), \quad \mu \in \mathscr{A}_t,
\end{align}
where $\mathscr{C}$ is some class of measurable subsets of $\Theta$. Finally, we introduce the following term. Let $\o{K}$ be a positive semidefinite Hilbert--Schmidt operator, we denote
\[
\varkappa(\o{K}) \eqdef \Bigl[\sqrt{\left(\tr(\o{K}^2) - \lmax(\o{K}^2) \right) \tr(\o{K}^2)} \Bigr]^{-1/2},
\quad 
\gamma(\o{K}) \eqdef \varkappa(\o{K})\tr(\o{K}).
\]

\begin{proposition}[GAR for $M$-estimators]
\label{prop:gar}
    Let the above assumptions hold. 
    Fix a continuously invertible operator $\A$ and denote
    \[
    \o{K} = \A \o{F}^{-1}\o{\Sigma}\o{F}^{-1}\A^*,
    \quad
    \hat{\o{K}} = \A \o{F}^{-1}\hat{\o{\Sigma}}\o{F}^{-1}\A^*.
    \]
    Then for any $\x >0$ such that $\eps_{F}(\x) \le 1/2$ it holds
    \begin{align}\label{eq:bound_GAR_I_condition}
        &\sup_{z> 0 }\abs*{\P\left\{\norm{\A(\theta_\mu - \theta_*)} > z \right\} - \P\left\{ \norm*{\o{AF}^{-1}G} > z \right\} } \\
        & \le \eps_G + \CONST e^{-\x} + \CONST \gamma(\o{K})\cdot \kappa(\o{AF}^{-1/2})\eps_{F}(\x).
    \end{align}
    Moreover, for any $\x, \t > 0$ such that $\hat{\eps}_{F}(\x, \t)\le 1/2$ and $\mu \in \mathscr{A}_{\t}$ such that $r_\mu \le 1/2$, we get
    \begin{align}\label{eq:bound_GAR_II_condition}
        &\sup_{z> 0} \abs*{\P\left\{\norm{\A(\theta_\mu - \theta_{\hat{\mu}})} > z \middle| \mu\right\} - \P\left\{ \norm*{\o{AF}^{-1}\hat{G} } > z \middle| \mu \right\} } \le \hat{\eps}_G(\t) + \CONST e^{-\x} \nonumber \\
        & + \CONST \gamma(\hat{\o{K}}) \kappa(\o{AF}^{-1/2})
        \Bigl[
        \hat{\eps}_{F}(\x, \t) + \frac{1}{\sqrt{\tr(\hat{\o{K}})}} (r_\mu + \hat{\eps}_F(\x, \t))\norm*{\o{A F}^{-1} g_\mu}
        \Bigr].
    \end{align}
\end{proposition}

This proposition follows directly from the general GAR.
\begin{lemma}[General GAR]
\label{lemma:GAR_general}
    Let $X, Y \in \R_+$ be random variables satisfying the following assumptions: 
    there exist constants 
    $m, \delta > 0$, $\rho \in \left[0, \frac{1}{2}\right]$ such that
    \begin{equation}
        \label{eq:comp_aux}
        \P\left(\abs{X - Y} \le \rho Y + m \right) \ge 1 - \delta ;
        \tag{GAR-I}
    \end{equation}
    there exists a centred Gaussian vector $G \sim \ND(0, \o{K})$ taking values in a Hilbert space $H$, and a constant $\Delta \in (0, 1)$, such that
    \begin{equation}
        \label{eq:gar_aux}
        \sup_{z > 0} \abs*{\P\left\{Y \le z\right\} - \P\left\{\norm{G}_H \le z\right\}} \le \Delta,
        \tag{GAR-II}
    \end{equation}
    with $\norm{\cdot}_H$ denoting the norm induced by the scalar product in $H$. 
    Then
    \[
    \sup_{z > 0} \abs*{\P\left\{X \le z\right\} - \P\left\{\norm{G}_H \le z\right\}} 
    \le \Delta + \delta + \CONST \gamma(\o{K}) \left(\frac{m}{\sqrt{\tr(\o{K})}} + \rho \right),
    \]
    with $\gamma(\o{K})$ defined by~\eqref{def:gamma0}.
\end{lemma}

The proof is in the Appendix~\ref{sec:main_gauss_approx}. The proposition follows directly.

\begin{proof}[Proof of Proposition~\ref{prop:gar}]
    To obtain the first result, we use Lemma~\ref{lemma:aux_g_nu}, setting $\nu = \mu_*$.
    This yields
    \begin{align*}
    \norm*{\A(\theta_\mu - \theta_*) + \A\o{F}^{-1} g_\mu} &\le \norm{\A\o{F}^{-1/2}} \cdot \frac{r_\mu}{1 - r_\mu} \cdot\norm{\o{F}^{-1/2}g_{\mu}} \\
    &\le \kappa(\o{AF}^{-1/2}) \cdot \frac{r_\mu}{1 - r_\mu} \cdot \norm{\o{AF}^{-1}g_{\mu}}.
    \end{align*}
    Taking into account Assumption~\eqref{asm:r}, we get that with high probability $r_\mu \le \eps_{F}(\x)$. Moreover, the assumption of the lemma ensures $\eps_{F}(\x) \le 1/2$. Thus, with probability at least $1-\CONST e^{-\x}$, we get
    \[
    \abs*{\norm{\A(\theta_\mu - \theta_*)} - \norm{\o{AF}^{-1} g_\mu}} \le \CONST \eps_{F}(\x) \norm{\o{AF}^{-1}g_\mu}.
    \]
    This ensures the validity of Assumption~\eqref{eq:comp_aux}. To validate \eqref{eq:gar_aux}, we consider Assumption \eqref{asm:gar_1} with a particular choice of $\mathscr{C}$. Specifically, we choose $\mathscr{C}$ to be a set of elliptical sets
    $
    \mathscr{C} \eqdef \Bigl\{ \left\{\eta \in \R^d \, : \, \norm*{\o{AF}^{-1} \eta} \le z \right\}  \, \Bigm| \, z> 0 \Bigr\}.
    $
    The claim follows.
    The proof of the second result is similar.
\end{proof}

The key theorem ensuring \eqref{eq:K_distance} relies on two more assumptions. First, to use Gaussian comparison, we assume that there exists a function $\eps_{K}(\x)$ such that for $\o{K}$ and $\hat{\o{K}}$ from Proposition~\ref{prop:gar}, it holds that
\begin{equation}
\label{asm:upxi}
   \P \left\{ \norm{\o{K} - \hat{\o{K}} }_1 \le \eps_{K}(\x) \right\} \ge 1-\CONST e^{-\x},
\end{equation}
where $\norm{\cdot}_1$ is $1$-Schatten norm. Finally, we assume that
\begin{equation}
\label{asm:concentration_g}
    \P \left\{\norm{\o{AF}^{-1}g_\mu} \le \eps_g(\x) \right\} \ge 1 -\CONST e^{-\x}.
\end{equation}

\begin{theorem}
\label{thm:bootstrap_M}
    Let $\t > 0$ be such that
    \begin{equation}
    \label{eq:bound_on_eps_K}
     \eps_{K}(\t) \le \frac{\tr(\o{K}^2) - \lmax(\o{K}^2)}{4\norm{\o{K}}}.
    \end{equation}
     With probability at least $1-\CONST e^{-\t}$
    \begin{align*}
        &\sup_{z> 0} \abs*{\P\left\{\norm{\A(\theta_\mu - \theta_{\hat{\mu}})} > z \middle| \mu\right\} - \P\left\{ \norm*{\A(\theta_\mu - \theta_*) } > z  \right\} } 
        \lesssim \eps_G + \hat{\eps}_G(\t) + \eps_K(\t)\\
        &+\inf_{\x \in \hat{\mathcal{X}}}  \Bigl[ e^{-\x} +
        \gamma({\o{K}}) \kappa(\o{AF}^{-1/2}) \Bigl[
        \hat{\eps}_{F}(\x, \t) + \frac{1}{\sqrt{\tr({\o{K}})}} (\eps_{F}(\t) + \hat{\eps}_F(\x, \t)) \eps_g(\t)
        \Bigr]
        \Bigr]\\
        &+ \inf_{\x \in \mathcal{X}}\Bigl[e^{-\x} + \gamma(\o{K})\cdot \kappa(\A \o{F}^{-1/2})\eps_{F}(\x)\Bigr],
    \end{align*}
    where $  \hat{\mathcal{X}} = \left\{\x:\, \hat{\eps}_{F}(\x, \t) \le 1/2 \right\}, 
    \quad
    \mathcal{X} = \left\{\x:\, \eps_{F}(\x) \le 1/2 \right\}.$
\end{theorem}

\begin{proof}
First, we consider~\eqref{eq:bound_GAR_II_condition}.
Assumptions~\eqref{asm:concentration_g} and~\eqref{asm:r} ensure that with probability at least $1-\CONST e^{-\t}$ it holds that
\[
\norm{\o{AF}^{-1}g_\mu} \le \eps_g(\t), 
\quad
r_\mu \le \eps_{F}(\t).
\]
Moreover, Lemma~\ref{lemma:bounds_kappa} and the bound \eqref{eq:bound_on_eps_K} ensure
$
\varkappa( \hat{\o{K}}) \le 2 \varkappa\left(\o{K} \right),
\quad 
\tr (\hat{\o{K}}) \le \frac{5}{4} \tr \left(\o{K} \right).
$
Thus, with probability at least $1-\CONST e^{-\t}$, we get
\begin{align}
\label{eq:bound_GAR_II_condition_upd}
    &\sup_{z> 0 }\abs*{\P\left\{\norm{\A(\theta_\mu - \theta_{\hat{\mu}})} > z \middle| \mu\right\} - \P\left\{ \norm*{\o{AF}^{-1}\hat{G} } > z \middle| \mu \right\} } \le \hat{\eps}_G(\t) + \CONST e^{-\x} \nonumber \\
    & + \CONST \gamma(\o{K}) \kappa(\o{AF}^{-1/2})
    \Bigl[
    \hat{\eps}_{F}(\x, \t) + \frac{1}{\sqrt{\tr(\o{K}})} (\eps_{F}(\t) + \hat{\eps}_F(\x, \t))\eps_g(\t)
    \Bigr].
\end{align}
Now, we recall \eqref{eq:bound_GAR_I_condition}. To get the result of the theorem, we have to bound 
\[
\abs*{
\P\left\{\norm*{\A \o{F}^{-1}G}_{\F} \le z\right\} - \P\left\{\norm*{\A \o{F}^{-1}\hat{G} }_{\F} \le z ~|~\mu\right\}}.
\]
Recall that $\A$ is self-adjoint. Corollary~2.3 by~\citet{gotze2017gaussian} ensures
\begin{multline*}
    \label{eq:4}
    \sup_{z\ge 0}\abs*{\P\left\{\norm*{\A \o{F}^{-1}G }_{\F} \le z\right\} - \P\left\{\norm*{\A \o{F}^{-1}\hat{G} }_{\F} \le z~|~\mu\right\}} \\
    \le \CONST 
    \left(\varkappa(\o{K}) + \varkappa(\hat{\o{K}})\right) 
    \norm*{\o{K} - \hat{\o{K}}}_1 \lesssim \varkappa(\o{K}) \eps_{K}(\t), 
\end{multline*}
where the last inequality holds with probability at least $1-\CONST e^{-\t}$.
To get the result, we combine this bound with \eqref{eq:bound_GAR_I_condition} and \eqref{eq:bound_GAR_II_condition_upd}.   
\end{proof}

Finally, to get an approximation for a suitably chosen distance $d$, one can use linearization.
The following scheme illustrates sketches a proof of \eqref{eq:K_distance},
\begin{align*}
    d(\theta_{*}, \theta_{\mu}) & \underset{(1)}{\overset{\mathrm{d}}{\approx}} \norm{\A (\theta_* - \theta_\mu)}_\F \underset{(2)}{\overset{\mathrm{d}}{\approx}} \norm{\o{AF}^{-1} G}_{\F} \underset{(3)}{\overset{\mathrm{d}}{\approx}} \norm{\o{AF}^{-1} \hat{G} }_{\F} \\
    & \underset{(2)}{\overset{\mathrm{d}}{\approx}} \norm{\A (\theta_\mu - \theta_{\hat{\mu}})}_{\F} \underset{(1)}{\overset{\mathrm{d}}{\approx}} d(\theta_\mu, \theta_{\hat{\mu}}).
\end{align*}
Step (1) is the linearization depending on the particular choice of the distance $d$. An example is the linearization of the Bures--Wasserstein distance in \eqref{eq:lin1} and~\eqref{eq:lin2}. Step (2) is GAR relying on Assumptions~\eqref{asm:r}, \eqref{asm:gar_1}, and~\eqref{asm:gar_2}. Step (3) is the Gaussian comparison result relying on Assumption~\eqref{asm:upxi}. Appendix~\ref{sec:boot_validity} illustrates the ideas developed in this section introducing the generalized bootstrap in the Bures--Wasserstein space.

\section{Experiments on graph-structured data}
\label{sec:exp}

The aim of this section is twofold. First, drawing on the ideas from \citet{haasler2024bures}, we demonstrate how the multiplier bootstrap performs on both synthetic and real graph-structured data, specifically related to brain connectomes. Second, we compare the approximating distribution constructed via multiplier bootstrap with the asymptotic distribution presented in Corollary~2.1 of the work \cite{kroshnin2019statistical}. The code supporting the experiments is available at \url{https://github.com/asuvor/bw_paper/}.

For completeness, we briefly recall the concept of asymptotic confidence sets. Corollary~2.1 in \cite{kroshnin2019statistical} ensures that
\begin{equation}
\label{eq:asm_conf_sets}
    \sqrt{n}\bw(B_{n}, B) \overset{\mathrm{d}}{\approx} \norm*{B^{1/2}_{n} \dT{B_n}{B_n}(Z)}_{\F}, ~~\text{as}~~n\rightarrow \infty,
\end{equation}
where $Z \sim \ND(0, \o{\Upxi}_n)$ with $\o{\Upxi}_n \eqdef \o{F}^{-1}_n \o{\Sigma}_n \o{F}^{-1}_n$ and 
\begin{equation}
\label{eq:F_n}
    \o{F}_n \eqdef -\frac{1}{n} \sumi \dT{B_n}{S_i}, \quad \o{\Sigma}_n \eqdef \frac{1}{n} \sumi\left(T^{S_i}_{B_n} - I \right) \otimes \left(T^{S_i}_{B_n} - I \right).
\end{equation}
To get an asymptotic $\alpha$-confidence set for the true barycenter $B$ from \eqref{eq:asm_conf_sets}, one uses the  quantile
\[
\tilde{z}^{\alpha} \eqdef \inf \left\{z > 0:\, \P\left(\norm*{B^{1/2}_{n} \dT{B_n}{B_n}(Z)}_{\F} \le z \right) \ge \alpha\right\}.
\]

\subsection{Computational complexity}
\label{sec:comp_eff}

We use the iterative algorithm proposed by~\cite{alvarez2016fixed} to compute the barycenters. 
It starts from $Q_0 \in \H_{++}(d)$, and its iteration is given by 
\begin{align}
\label{def:iterative_algorithm}
    Q_{k+1} = f_\mu[Q_k] \eqdef{}& (T_\mu[Q_k] + I) Q_k (T_\mu[Q_k] + I) \in \H_{++}(d), \nonumber \\
    T_{\mu}[Q] \eqdef{}& \int\limits_{\H_{++}(d)} (T^S_{Q} - I) \dd \mu(S),
\end{align}
where $\mu$ is a fixed probability measure on $\H_{++}(d)$ with finite second moment.
The computational complexity of the algorithm can be estimated as follows. 

\begin{theorem}[Computational complexity]
\label{thm:complexity}
    Let $\mu$ be a fixed probability measure on $\H_{++}(d)$ with finite second moment. Denote
    \[
    \V_\mu(Q) \eqdef \min_{Q \in \H_{++}(d)} \int\limits_{\H_{++}(d)} \bw^2(Q, S) \dd \mu(S).
    \]
    Set
    $
    \rho_\mu \eqdef \frac{1}{2\kappa^{7/2}_\mu}.
    $
    For all $k = 0, \dots, N$
    \[
    \V_\mu(Q_k) - \V_\mu(B_\mu) \le (1 - \rho_\mu)^k \left(\V_\mu(Q_0) - \V_\mu(B_\mu)\right).
    \]
    Moreover, for the given precision $\eps>0$ it is enough to make $k \ge N$ steps with
    \[
    N = \frac{1}{\rho_\mu} \ln
    \left(
    \frac{1}{\eps}\cdot \frac{2\kappa^{1/2}_\mu}{\lmin(\o{F}_{\mu})}\left(\V_\mu(Q_{0}) - \V_\mu(B_\mu)   \right)^{1/2}  
    \right), 
    \quad 
    \o{F}_{\mu} \eqdef -\int\limits_{\H_{++}(d)} \dT{B_\mu}{S}\dd \mu(S).
    \]
\end{theorem}

Appendix~\ref{appx:computation} contains the proof.

\begin{remark}
    The paper~\cite{altschuler2021averaging} investigates the convergence rate of \eqref{def:iterative_algorithm} in the OT setting, i.e., only for real-valued matrices. In this case the contraction constant $\rho_\mu = \frac{1}{\kappa^{3/2}_{\mu}}$. Upon closer examination of the proof techniques used in this paper, one can conjecture that the complex-valued case reduces to the real-valued one, in which case the factor $\rho$ would be improved. Nevertheless, this hypothesis requires verification. 
\end{remark}

\paragraph{Complexity of bootstrap approximation.} Let $I$ denote the average number of iterations in the iterative algorithm. Theorem~\ref{thm:complexity} provides an upper bound on $I$. However, in practice, the iterative algorithm requires fewer steps.
Computing the barycenter of $n$  matrices requires $O(I \cdot n \cdot \mathcal{K}(d))$ operations, where $\mathcal{K}(d)$ is the complexity of matrix operations (matrix inversion and matrix square root computations). The best-known complexity for matrix inversion is approximately $O(d^{2.38})$. Moreover, to the best of our knowledge, the complexity of computing the square root of a matrix is $O(d^3)$. Therefore, $K(d) = O(d^3)$, resulting in a total computational complexity of $O(I \cdot n \cdot  d^3)$. Thus, the computational complexity of the multiplier bootstrap is $O(M\cdot I \cdot n \cdot d^3)$, where $M$ is the number of resamplings.

\paragraph{Complexity of asymptotic approximation.} To measure the computational complexity of estimating the asymptotic distribution, we note that the operator $\dT{B}{S}(X)$ admits an explicit representation; see Lemma~A.2 by \cite{kroshnin2019statistical}. Specifically, computing each entry in its matrix representation requires $O(d^2)$ operations. Since the space dimension is $\frac{d (d+1)}{2}$, the total complexity of constructing the matrix representation is $O(d^2 \cdot(d^2)^2) = O(d^6)$. Therefore, computing the representation of $\o{F}_n$ (see \eqref{eq:F_n}) requires $O(n \cdot d^6)$ operations. Additionally, sampling a Gaussian matrix $Z$ involves $O(d^4)$ operations. Thus, the total complexity can be estimated as $O(n\cdot d^6 + M \cdot d^4)$, where $M$ is the number of resamplings. Therefore, for large $d$, estimating the asymptotic distribution can be significantly more resource-intensive compared to the bootstrap method.

\subsection{Bures--Wasserstein barycenters of graphs}
\label{sec:bw_graphs}

\cite{haasler2024bures} proposed a novel framework for defining and computing the mean of a set of graphs using the Bures--Wasserstein distance. In the following, we adhere to this setting and present it for completeness.

The authors focus on aligned graphs, meaning graphs with the same number of nodes, with each node corresponding to a specific node in the other graphs. For instance, each vertex might represent a specific area of the head where an electrode is placed to capture EEG signals. Section~\ref{sec:experiments_on_real_data} introduces this setting in more detail.

Let $G$ be an undirected weighted graph with $d$ nodes without self-loops. 
In the following, we assume the weights to be positive.
The adjacency matrix and degree matrix of $G$ are denoted as $A_G$ and $D_G$, respectively.
The graph Laplacian of $G$ is defined as $L \eqdef D_G - A_G.$
Denote by $\mathscr{G}(d)$ the set of aligned positive-weighted and connected graphs with $d$ nodes. The Bures--Wasserstein distance between $G_1 \in \mathscr{G}$ and $G_2 \in \mathscr{G}(d)$ is
\[
\bw_{\mathscr{G}}(G_1, G_2) = \bw(L^{\dag}_{1}, L^{\dag}_{2}),
\]
where $L^{\dag}_{1}$, $L^{\dag}_{2}$ are the pseudo-inverses of their graph Laplacians.

Consider a population of graphs $G_1, \dots, G_n \in \mathscr{G}(d)$. Let the corresponding graph Laplacian be $L_1, \dots, L_n$. 
The authors reduce the problem of finding the barycenter of the graphs to the problem of finding the barycenter of their inverted graph Laplacians. Since all $G_i$ are connected, $L_1, \dots, L_n$ share the same kernel, $\mathrm{span}(\mathbf{1}_{d})$, where $\mathbf{1}_{d} \in \R^d$ is the vector of all ones.
Thus, it suffices to restrict the Laplacians to the orthogonal complement of the kernel and then compute the barycenter. In what follows, we denote restricted inverse graph Laplacians as
\begin{equation}
\label{def:observations}
    S_i \eqdef U_{\mathbf{1}_{d}}^\top L_i^{\dag} U_{\mathbf{1}_{d}} 
    = \left(U_{\mathbf{1}_{d}}^\top L_i U_{\mathbf{1}_{d}}\right)^{-1}, 
\end{equation}
with $U_{\mathbf{1}_{d}} \in \R^{d \times (d-1)}$ being a matrix of an orthonormal basis on $\mathrm{span}(\mathbf{1}_{d})^{\perp} $.
By construction, $S_i \in \H_{++}(d-1)$.

In many practical problems, observed graphs are supposed to be i.i.d., $G_i \overset{\mathrm{iid}}{\sim} P_G$. Consequently, the corresponding graph Laplacians $L_i$ and their inverted restrictions $S_i$ are i.i.d.\ ($S_i \overset{\mathrm{iid}}{\sim} P$).

\subsection{Experiments on Weighted stochastic block model (WSBM)}

We use WSBM data to compare non-asymptotic and asymptotic confidence sets. Each generated graph $G$ has $d$ nodes divided into two non-overlapping groups (communities). The size of each group is random: the first group contains $d_1 = \frac{d}{2} - \mathrm{Unif}\{-2, 2\}$ nodes, and the second group contains $d_2 = d - d_1$ nodes. The corresponding adjacency matrix $A\in \R^{d \times d}$ has a block structure and the  entries within each block are i.i.d.\ Poisson r.v.,  
\[
A = \begin{pmatrix}
    C_{11} & C_{12} \\
    C_{21} & C_{22}
\end{pmatrix},
\quad
a \sim \begin{cases}
    \mathrm{Po}(20) &~~\text{for}~~a \in C_{11},\\
    \mathrm{Po}(15) &~~\text{for}~~a \in C_{22},\\
    \mathrm{Po}(6) &~~\text{for}~~a \in C_{12}, C_{21}.
\end{cases}
\]
The blocks $C_{11}$ and $C_{22}$ represent intra-group connections and $C_{12} = C^{T}_{21}$ represents inter-group connections. The probabilities of observing a non-zero edge between each pair of nodes within the corresponding blocks are $p_{11} = 0.8$, $p_{22} = 0.7$, $p_{12} = p_{21} = 0.3$. 
To ensure that a generated graph is connected, we consider a regularized adjacency matrix $A + \varepsilon E$, with $E$ being a $d\times d$ matrix of all ones. In the experiments, we set $\varepsilon = 1$.

We illustrate the method using dimension $d = 40$. The simulated population contains $N = 8 000$ $d\times d$ WSBM adjacency matrices with inverse graph Laplacians $S_i$ as in \eqref{def:observations}. We compute the population barycenter $B$ from the entire population. To estimate the empirical barycenter $B_n$, we use different sample sizes $n \in \{10, 30, 100\}$. 

Specifically, to estimate the empirical cumulative distribution function (ECDF) of \\
$\sqrt{n}\bw(B_n, B)$, we subsample 
$n$ observations with replacement from the entire population. For estimating the ECDFs of $\sqrt{n}\bw(B_n, B_w)$, we employ the multiplier bootstrap method and set the bootstrap weights to be Poisson, $w_i \sim \mathrm{Po}(1)$. 

For each sample size $n$ and each case, we generate 100 independent curves. Finally, to evaluate the quality of the approximations provided by multiplier bootstrap and asymptotic result, we compute the Kolmogorov distance between the ECDF of $\sqrt{n}\bw(B, B_n)$ and the realizations of $\sqrt{n}\bw(B_n, B_w)$ and $\norm{B^{1/2}_n\dT{B_n}{B_n}(Z)}$, respectively. The lower panel illustrates the distributions of the Kolmogorov distances: the light-blue curves correspond to the bootstrap case, while the orange curves represent the asymptotic case. Fig.~\ref{fig:SBM_d_40} presents the result. 
\begin{figure}[!ht]
    \centering
    \includegraphics[width=250pt]{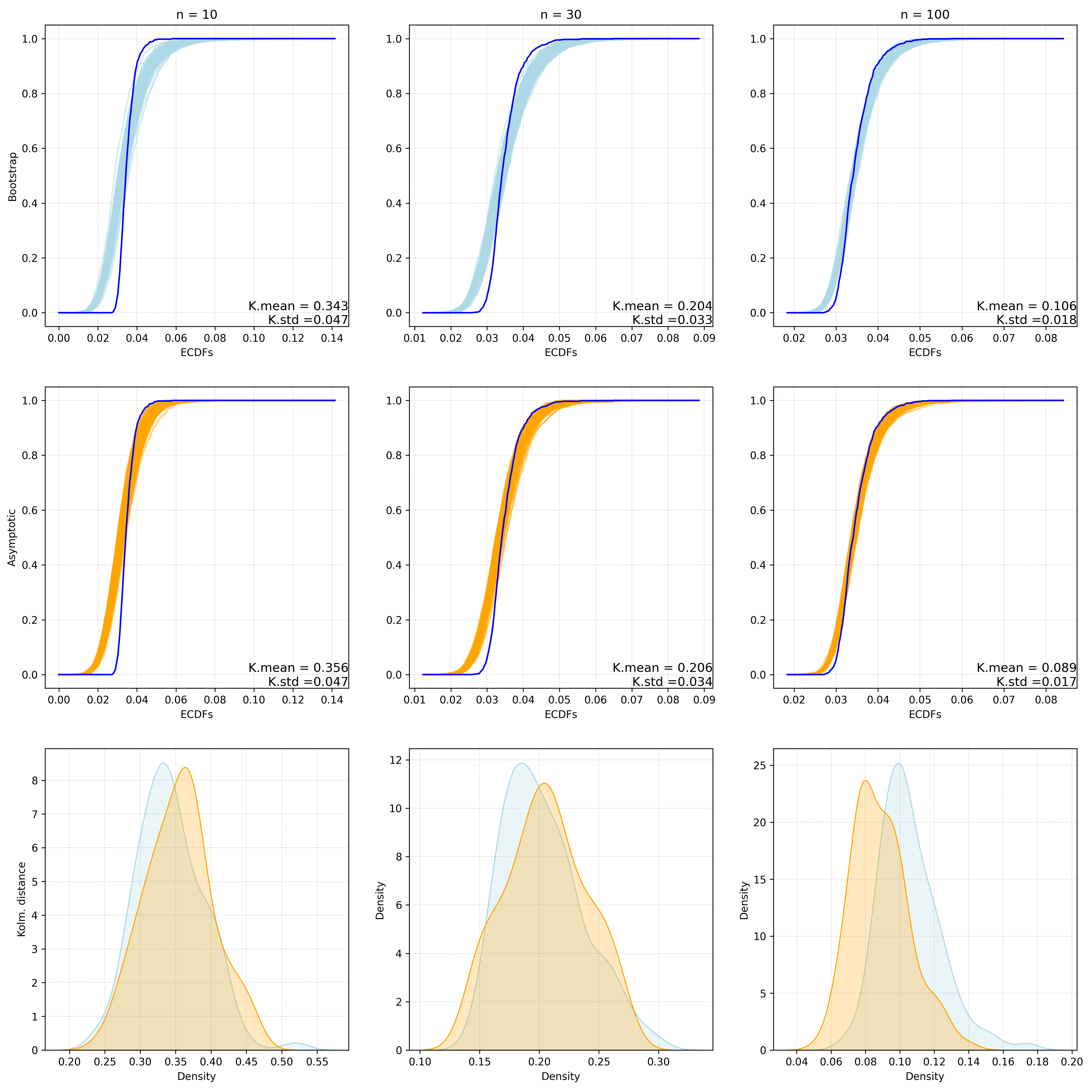}
    \caption{\textbf{WSBM, $d = 40$.} ECDFs of 
    $\sqrt{n}\,\bw(B,B_n)$ (dark blue), 
    $\sqrt{n}\,\bw(B_n,B_w)$ (light blue), and 
    $\bigl\|B_n^{1/2}\,\dT{B_n}{B_n}(Z)\bigr\|_{\F}$ (orange). The mean Kolmogorov distance to the true distribution and its standard deviation are in the lower‐right corner.  The lower panels show the distribution of the Kolmogorov distance between the ``true'' ECDF and each approximation: bootstrap vs.\ ``true'' (light blue) and asymptotic vs.\ ``true'' (orange). }
    \label{fig:SBM_d_40}
\end{figure}

\subsection{Experiments on connectomes}
\label{sec:experiments_on_real_data}

The EEGBCI dataset contains EEG recordings from $64$ electrodes from $109$ participants. Each participant completed $14$ sessions, corresponding to a distinct task associated with imagined movements. Each electrode captures electrical activity from a particular region of the scalp and the underlying brain regions when a person fulfills the task. From these recordings, we construct functional connectomes. Functional connectomes are networks that show how different brain regions connect and interact. Nodes in the network correspond to brain regions. Edge weights represent the interactions between these regions. To construct the connectomes, we used EEG signals from 3 tasks (imagining moving the left hand and the feet). The edge weight between two nodes is the envelope correlation between the EEG signals from the corresponding pairs of electrodes. Thus, we got $109$ connectomes of size $64\times 64$. We convert them to projected graph Laplacians as described in Section~\ref{sec:bw_graphs}.  Using the entire population, we compute the true barycenter $B$. To estimate the distribution of $\sqrt{n} \bw(B, B_n)$ (for $n = 10, 50, 70$), we sample with replacement from the population. To estimate the distribution of $\sqrt{n} \bw(B_n, B_w)$, we employ the multiplier bootstrap with Poisson weights $w_i \sim \mathrm{Po}(1)$. Fig.~\ref{fig:real_data} presents the result.

\begin{figure}[!ht]
    \centering
    \includegraphics[width=250pt]{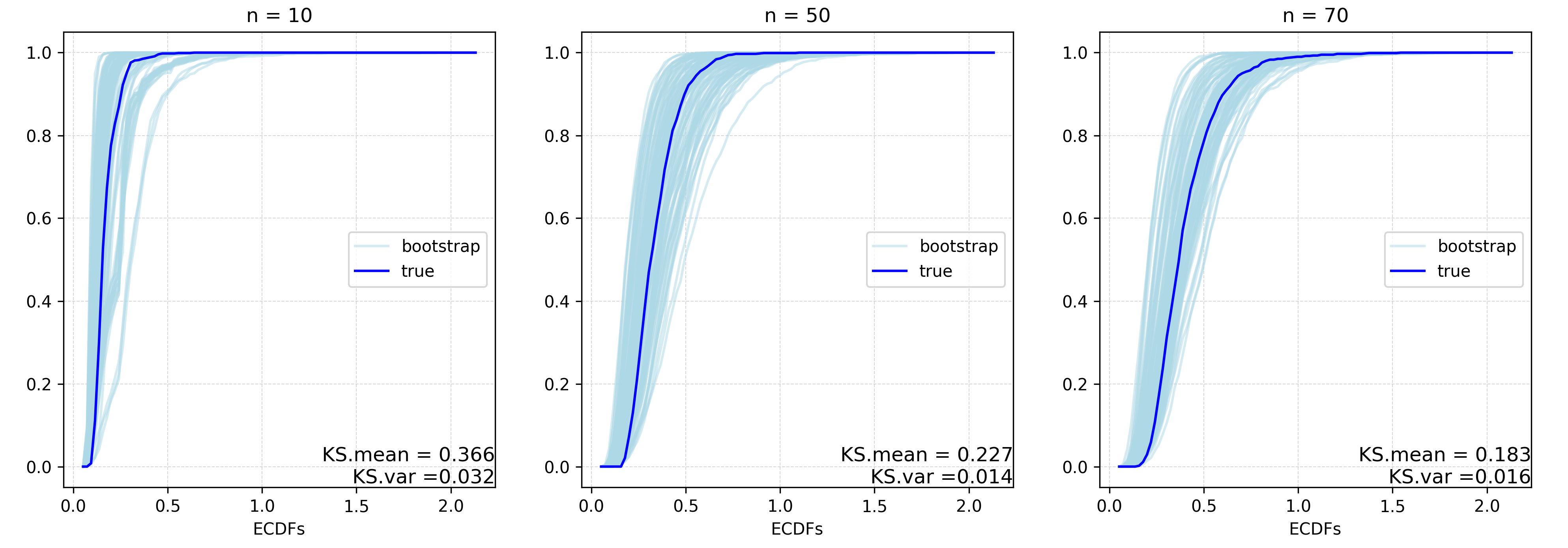}
    \caption{\textbf{EEGBCI, $d = 64$}. Empirical distribution functions for $\sqrt{n} \bw(B_n, B_w)$  (light-blue) and $\sqrt{n} \bw(B_n, B)$ (dark-blue).}
    \label{fig:real_data}
\end{figure}

\bibliographystyle{plainnat} 
\bibliography{references.bib}    

\begin{appendix}
\section{Approximation bounds in the Bures--Wasserstein space}
\label{sec:bw_geometry}

We begin collecting some facts that are crucial for generalized bootstrap validation. Recall that the Fr{\'e}chet differentiabity implies
\[
    T^{S}_{Q + X} = T^{S}_{Q} + \dT{Q}{S}(X) + o(\norm{X})
    \quad
    \text{as $\norm{X} \to 0$,} \quad X \in \H(d),
\]
where $\dT{Q}{Q} \colon \H(d) \to \H(d)$ is a negative semi-definite operator. 
For completeness, we provide Lemma~A.3 by \citet{kroshnin2019statistical}.

\begin{statement}
\label{lemma:dT}
    For any $S \in \H_+(d)$, $Q \in \H_{++}(d)$, the operator $\dT{Q}{S}$ satisfies the following properties:
    \begin{enumerate}[(I)]
        \item it is self-adjoint;
        
        \item it is negative semi-definite;
        
        \item it enjoys the following bounds:
        \begin{align*}
            - \left\langle \dT{Q}{S}(X), X \right\rangle &
            \le \frac{\lmax^{1/2}\left(S^{1/2} Q S^{1/2}\right)}{2} \norm{Q^{- 1 / 2} X Q^{- 1 / 2}}_{\F}^2,\\
            - \left\langle \dT{Q}{S}(X), X \right\rangle &
            \ge \frac{\lmin^{1/2}\left(S^{1/2} Q S^{1/2}\right)}{2} \norm{Q^{- 1 / 2} X Q^{- 1 / 2}}_{\F}^2;
        \end{align*}

        \item it is homogeneous w.r.t.\ $Q$ with degree $-\frac{3}{2}$ and w.r.t.\ $S$ with degree $\frac{1}{2}$, i.e., $\dT{a Q}{S} = a^{-3 / 2} \dT{Q}{S}$ and $\dT{Q}{a S} = a^{1 / 2} \dT{Q}{S}$ for any $a > 0$;
        
        \item it is monotone w.r.t.\ $S^{1/2} Q S^{1/2}$ (once range $S$ is fixed): $\dT{Q_0}{S_0} \preccurlyeq \dT{Q_1}{S_1}$ in the sense of self-adjoint operators on $\H(d)$ whenever $S_0^{1/2} Q_0 S_0^{1/2} \preccurlyeq S_1^{1/2} Q_1 S_1^{1/2}$ and $\range(S_0) = \range(S_1)$; in particular, $\dT{Q}{S}$ is monotone w.r.t.\ $Q \in \H_{++}(d)$ for fixed $S$.
    \end{enumerate}
\end{statement}

Let $Q \in \H_{++}(d)$ and define
\begin{equation*}
\label{def:AQ}
    \A_Q \eqdef \left(- \frac{1}{2} \dT{Q}{Q}\right)^{1/2}.
\end{equation*}
Lemma~A.3 by~\citet{kroshnin2019statistical} ensures its existence.

\begin{lemma}[Properties of of $\A_Q$]
\label{lemma:spec_A}
    The following equalities hold  
	\begin{equation}
	    \label{def:bounds_A}
    	\norm*{\A_Q} = \frac{1}{2 \sqrt{\lmin(Q)}}, \quad
	    \norm*{\A^{-1}_Q} = 2 \sqrt{\lmax(Q)}.
	\end{equation}
    
    Moreover, there exists a unitary operator $\o{U}_Q$ on $\H(d)$ such that for any $X \in \H(d)$ holds
    \begin{equation}
        \label{eq:nice_eq_A}
        (\o{U}_Q \A_Q) X = Q^{1/2} \dT{Q}{Q}(X).
    \end{equation}
\end{lemma}

\begin{proof}
    First we prove~\eqref{eq:nice_eq_A}. Without loss of generality, 
    let $Q$ be a diagonal matrix, i.e., $Q = \diag(q_1, \dots, q_d)$. 
    It is enough to consider a diagonal $Q$, because for any unitary $U$
    \[
    T^{USU^*}_{USU^*} = U T^{S}_{Q} U^*.
    \]
    
    Using the explicit expression for $\dT{Q}{Q}(X)$ (see formula~(A.2) by~\citet{kroshnin2019statistical}), we get
	\begin{align*}
    	- \left\langle \dT{Q}{Q} (X), X \right\rangle 
    	&= \sum_{i, j = 1}^d \frac{X_{ij}}{q_i + q_j} X_{ij}
    	= \sum_{i, j = 1}^d (q_i + q_j) \left(\frac{X_{ij}}{q_i + q_j}\right)^2 \\
    	& = 2 \sum_{i, j = 1}^d \left(\sqrt{q_i} \frac{X_{ij}}{q_i + q_j}\right)^2
    	= 2 \norm*{Q^{1/2} \dT{Q}{Q}(X)}_{\F}^2.
	\end{align*}
	Then $\norm*{\A_Q(X)}_{\F} = \norm*{Q^{1/2} \dT{Q}{Q}(X)}_{\F}$. Thus, these operators are unitary equivalent.
    
    Now we prove~\eqref{def:bounds_A}. The above chain of equations ensures 
	$\norm*{\A_Q(X)}_{\F}^2 
	= \frac{1}{2} \sum_{i, j = 1}^d  \frac{X^2_{ij}}{q_i + q_j}.$
    This yields
	\begin{equation*}
    	\frac{1}{4 \lmax(Q)} \norm*{X}_{\F}^2 
    	\le \norm*{\A_Q(X)}_{\F}^2 
    	\le \frac{1}{4 \lmin(Q)} \norm*{X}_{\F}^2.
	\end{equation*}
 
	One can show in the same way as in the proof of Corollary~A.2 by \citet{kroshnin2019statistical} that these inequalities are sharp. The result follows immediately.
\end{proof}

We will often quantify the closeness of $S \in \H_{++}(d)$ and $Q\in \H_{++}(d)$ as
\[
r(Q, S) \eqdef \norm*{Q^{-1/2}S Q^{-1/2}-I},
\]
with $\norm{\cdot}$ being the operator norm and $I$ standing for the $d \times d$ identity matrix.

\begin{lemma}[Local Lipschitz continuity of $\A_Q$]
\label{lemma:loc_Lip}
    Let $B, Q \in \H_{++}(d)$. 
    If $r(B, Q) \le 1/2$, 
    \[
    \norm{\A_B - \A_Q} \le r(B, Q) \cdot \norm{\A_B}. 
    \]
\end{lemma}

\begin{proof}
    Let $Q' = B^{-1/2}Q B^{-1/2}$.
    Lemma~\ref{lemma:dT} ensures that the mapping $Q \mapsto \dT{Q}{Q}$ is monotone and $(-1)$-homogeneous.
    Then $Q \mapsto \A_Q$ is antimonotone and $(-\frac{1}{2})$-homogeneous. This entails
    \[
        \left(1 - \frac{1}{2} r(B, Q)\right) \A_B \preccurlyeq \frac{1}{\sqrt{\lmax(Q')}} \A_B
        \preccurlyeq \A_Q 
        \preccurlyeq \frac{1}{\sqrt{\lmin(Q')}} \A_B \preccurlyeq \left(1 + r(B, Q)\right) \A_B ,
    \]
    what yields the result.
\end{proof}

Now we are to establish a connection between $\bw(Q, S)$ and the Frobenius norm of the difference $\norm*{Q-S}_{\F}$. 

\begin{lemma}
\label{lemma:A_Frob_norm_new}
    Let $Q, S \in \H_{+}(d)$ be such that $r(B, Q) \le 1/2$ and
    $r(B, S) \le 1/2$.
    Then
    \begin{equation*}
        \label{eq:approx_BW_new}
        \abs*{\frac{\bw(Q, S)}{\norm*{\A_{B} (Q - S)}_{\F}} - 1} \le 4 r(B, Q) + 2 r(B, S).
    \end{equation*}
\end{lemma}

\begin{proof}
    For simplicity, we denote
    \[
    S'_Q \eqset Q^{-1/2} S Q^{-1/2}, 
    \quad
    Q'_B \eqset B^{-1/2}QB^{-1/2},
    \quad
    S'_B \eqset B^{-1/2}SB^{-1/2}.
    \]

	Lemma~A.6 by \citet{kroshnin2019statistical} ensures
	\begin{align*}
    	&-\frac{2}{\left(1 + \lmax^{1/2}(S'_Q)\right)^2}  \left\langle \dT{Q}{Q} (S - Q), S - Q \right\rangle 
    	\le \bw^2(S, Q) \\
        &\le
    	-\frac{2}{\left(1 + \lmin^{1/2}(S'_Q)\right)^2} \left\langle \dT{Q}{Q} (S - Q), S - Q \right\rangle.
	\end{align*}
	
    Due to the monotonicity and homogeneity of the operator $\dT{Q}{S}$ (see (IV) and (V) in Lemma~\ref{lemma:dT}), it holds that
	\begin{align*}
    	\dT{Q}{Q}  \preccurlyeq  \dT{\lmax(Q'_B) B}{\lmax(Q'_B) B}  = \frac{1}{\lmax(Q'_B)} \dT{B}{B}, \quad
    	\dT{Q}{Q}  \succcurlyeq \dT{\lmin(Q'_B) B}{\lmin(Q'_B) B} = \frac{1}{\lmin(Q'_B)} \dT{B}{B}.
	\end{align*}
 
	Combining these inequalities with \eqref{eq:nice_eq_A}, we get
	\begin{align}
	    \label{eq:bound_dbw_A}
	    \frac{4 \norm*{\A_B(S - Q)}^2_F}{\lmax(Q'_B)\left(1 + \lmax^{1/2}(S'_Q)\right)^2}
    	\le \bw^2(S, Q) 
    	\le \frac{4 \norm*{\A_B(S - Q)}^2_F}{\lmin(Q'_B)\left(1 + \lmin^{1/2}(S'_Q)\right)^2}.
	\end{align}
     
    The last step is to get the bounds on $\lmin(Q'_B)$ and $\lmax(Q'_B)$. 
    Let $r_Q \eqset r(B, Q),
    r_S \eqset r(B, S)$. It holds
    \begin{equation*}
    1 - r_Q \le \lmin(Q'_B) \le \lmax(Q'_B) \le 1 + r_Q.
    \end{equation*}
 
    Assumption $r_Q \le \frac{1}{2}$ yields
    \[
        \lmax^{-1/2}(Q'_B) \ge 1 - \frac{1}{2} r_Q, 
        \quad 
        \lmin^{-1/2}(Q'_B) \le 1 + 2 r_Q.
    \]
 
	Further, assumptions $r_Q \le \frac{1}{2}$ and $r_S \le \frac{1}{2}$ yield
	\begin{equation*}
	    \lmin(S'_Q) \ge \tfrac{\lmin(S'_B)}{\lmax(Q'_B)} \ge 1 - r_Q - r_S, \quad
	    \lmax(S'_Q) \le \tfrac{\lmax(S'_B)}{\lmin(Q'_B)} \le 1 + 2 r_Q + 2 r_S.
	\end{equation*}

	Then 
	\begin{equation*}
        \tfrac{2}{1 + \lmax^{1/2}(S'_Q)} \ge 1 - \frac{1}{2} r_Q - \frac{1}{2} r_S, \quad
    	\tfrac{2}{1 + \lmin^{1/2}(S'_Q)} \le 1 + r_Q + r_S.
	\end{equation*}
	Thus, we obtain
    \begin{align*}
        &2 \lmax^{-1/2}(Q'_B)\left(1 + \lmax^{1/2}(S'_Q) \right)^{-1} \ge 1 - r_Q - \frac{1}{2} r_S,\\
        &2 \lmin^{-1/2}(Q'_B)\left(1 + \lmin^{1/2}(S'_Q) \right)^{-1} \le 1 + 4 r_Q + 2 r_S .
    \end{align*}
    Combining these inequalities with~\eqref{eq:bound_dbw_A}, we get the result.
\end{proof}

The following lemma connects $B_{\mu}$ and $T_{\mu}$. Let $\o{F}$ be some fixed positive-definite operator acting from $\H(d)$ to $\H(d)$. Denote
\begin{equation}
\label{def:rho}
    r \eqdef r(B, B_{\mu}) + r(\o{F}, \o{F}_{\mu}),
    \quad
    \rho \eqdef 2 \sqrt{\kappa(\o{F})} r,
\end{equation}
with $\kappa(X) = \norm{X} \cdot \norm{X^{-1} }$ being the condition number of $X$.

By analogy with the barycenter mapping $\mathscr{B}$ \eqref{def:bar_map}, we define the $\mathscr{T}$-mapping and $\mathscr{F}$-mapping. Let $\mu' \in \mathcal{M}_{2}\bigl(\H_{++}(d)\bigr)$. We set
\begin{gather}
\label{def:T_mu}
    \mathscr{T}:\mu' \mapsto T_{\mu'}\eqdef \underset{\H_{++}(d)}{\int} \left(T^{S}_{B} - I \right)d\mu'(S);,
    \quad 
    \mathscr{F}: \mu' \mapsto \o{F}_{\mu'} \eqdef -\underset{\H_{++}(d)}{\int} \dT{B}{S}\dd \mu'(S), 
\end{gather}
with $T^{S}_{B}$ and $\dT{B}{S}$ coming from \eqref{def:dTQS}.

\begin{lemma}
\label{lemma:Qw_expansion}
    If $r \le \frac{1}{2}$. Then, the following approximations hold:
    \begin{gather}
    \label{bound:B_OT_F}
    \frac{\norm*{B_{\mu} - B -  \o{F}^{-1} \Tmw}_{\F}}{\norm*{\o{F}^{-1} \Tmw}_{\F}}  \le \rho,\\
    \label{bound:B_OT_BW}
    \abs*{\frac{\bw(B_{\mu}, B)}{\norm*{\A \o{F}^{-1} \Tmw}_{\F}} - 1}
    \le 3 \sqrt{\kappa(B)} \rho.
    \end{gather}
\end{lemma}

\begin{proof}
    First, we introduce an auxiliary operator $\o{D}_{\mu}$. 
    Let $B_{t} = tB_{\mu} + (1-t)B$, $t\in [0, 1]$. We set
    \begin{equation}
    \label{def:D_mu}
    \o{D}_{\mu} \eqdef - \underset{\H_{++}(d)}{\int}  \left[ \int\limits_{0}^{1} \dT{B_t}{S}\dd t\right]d\mu(S).
    \end{equation}

    \paragraph*{Proof of \eqref{bound:B_OT_F}} 
    We write the Taylor expansion for $B_{\mu}$ in the neighbourhood of $B$ in integral form (see Theorem~2.2 by~\citet{kroshnin2019statistical}),
    $
    B_\mu - B = \o{D}^{-1}_{\mu} T_\mu.
    $
    This ensures $B_{\mu} - B - \o{F}^{-1} \Tmw = \left(\o{D}_{\mu}^{-1} \o{F} - \Id\right) \o{F}^{-1} \Tmw,$
    with $\o{I}$ being the identity operator. We set $B_{\Delta} \eqset B_{\mu} - B$ and get
	\begin{equation*}
    	\frac{\norm*{B_{\Delta} - \o{F}^{-1} \Tmw}_{\F}}{\norm*{\o{F}^{-1} \Tmw}_{\F}} 
    	\le \norm*{\o{D}_{\mu}^{-1} \o{F} - \Id}.
	\end{equation*}
    The bounds on $\o{D}_{\mu}$ from Lemma~\ref{lemma:D_bounds} yield 
    $
    \left(1 - r\right) \o{F}^{-1}
    \preccurlyeq 
    \o{D}^{-1}_{\mu}
    \preccurlyeq
    \left(1 + 2r\right) \o{F}^{-1}.
    $
    Therefore,
    \[
    \norm*{\o{D}_{\mu}^{-1} \o{F} - \Id}
    \le \sqrt{\kappa(\o{F})}
    r(\o{D}_{\mu}^{-1}, \o{F}^{-1})
    \le \rho.
    \]
    The claim follows. To prove \eqref{bound:B_OT_BW}, 
    we use Lemma~\ref{lemma:A_Frob_norm_new} and set $ Q = B$, $S = B_{\mu}$. This yields
    \[
    \abs*{\frac{\bw(B_{\mu}, B)}{\norm*{\A B_{\Delta}}_{\F}} - 1} \le 2 r_B.
    \]

    Combining the above line of reasoning with the triangle inequality, we get
    \begin{align*}
    	\abs*{\frac{\norm*{\A B_{\Delta}}_{\F}}{\norm*{\A \o{F}^{-1} \Tmw}_{\F}} - 1}
    	\le \frac{\norm*{\A\left(B_{\Delta} - \o{F}^{-1} \Tmw\right)}_{\F}}{\norm*{\A \o{F}^{-1} \Tmw}_{\F}} 
        \le \kappa(\A) \frac{\norm*{B_{\Delta} - \o{F}^{-1} \Tmw}_{\F}}{\norm*{\o{F}^{-1} \Tmw}_{\F}} 
    	\overset{\text{by~}\eqref{bound:B_OT_F}}{\le} \sqrt{\kappa(B)} \rho.
    \end{align*}
    Note that the last inequality holds due to $\kappa(\A) = \sqrt{\kappa(B)}$ (see Lemma~\ref{lemma:spec_A}). 
    Combining the above bounds, we get
	\begin{align*}
        \abs*{\frac{\bw(B_{\mu}, B)}{\norm*{\A \o{F}^{-1} \Tmw}_{\F}} - 1} 
        &\le 2 r_B 
        + \left(1 + 2r_B\right) \abs*{\frac{\norm*{\A B_{\Delta}}_{\F}}{\norm*{\A \o{F}^{-1} \Tmw}_{\F}} - 1}  \\
        &\le 2 r_B + 2 \abs*{\frac{\norm*{\A B_{\Delta}}_{\F}}{\norm*{\A \o{F}^{-1} \Tmw}_{\F}} - 1} \le 3 \sqrt{\kappa(B)} \rho.
	\end{align*}
     The second and the third inequalities rely on $r \le \frac{1}{2}$ and $2 r_B \le \rho$, respectively.
\end{proof}

Now we fix some $\hat{\mu} \in \mathcal{M}\bigl(\H_{++}(d)\bigr) $ and define
\begin{equation}
\label{def:rho_hat}
    \hat{r} \eqdef r(B, B_{\hat{\mu}}) + r(\o{F}, \o{F}_{\hat{\mu}}),
    \quad
    \hat{\rho} \eqdef 2\sqrt{\kappa(\o{F})}~\hat{r}.
\end{equation}
    
\begin{corollary}
\label{def:corollary_bounds_F_BW}
    If $r \le \frac{1}{2}$ and $\hat{r} \le \frac{1}{2}$, then the following bounds hold
    \begin{align}
    \label{def:corollary_bounds_F}
        &\norm*{B_{\hat{\mu}} - B_{\mu} - \o{F}^{-1} (T_{\hat{\mu}} - T_{\mu})}_{\F} \le \hat{\rho} \norm*{\o{F}^{-1} (T_{\hat{\mu}} - T_{\mu})}_{\F} 
        + (\rho + \hat{\rho}) \norm*{\o{F}^{-1} T_{\mu}}_{\F},
    \end{align}
    \begin{align}\label{def:corollary_bounds_BW}
        &\abs*{\bw(B_{\hat{\mu}}, B_{\mu}) - \norm{\A \o{F}^{-1} (T_{\hat{\mu}} - T_\mu)}_{\F}} \\ 
        &\qquad \le 6 \kappa(\A) \left(\hat{\rho} + \rho \right) \norm{\A\o{F}^{-1} \left(T_{\hat{\mu}} - T_{\mu} \right) }_{\F}
        + 4 \left(\hat{\rho} + \rho\right)\norm{\A}
        \norm{\o{F}^{-1} T_{\mu}}_{\F}.\nonumber
    \end{align}
\end{corollary}

\begin{proof}
Claim~\eqref{def:corollary_bounds_F} follows directly from
\eqref{bound:B_OT_F}. Next, we prove \eqref{def:corollary_bounds_BW}. 
For the moment, we set
$\Delta \eqset B_{\mu} - B_{\nu} - \o{F}^{-1}\left(T_{\mu} - T_{\hat{\mu}}\right),$
and $\hat{r}_B \eqset r(B, B_{\hat{\mu} })$, $
\hat{r}_{F} \eqset r(\o{F}, \o{F}_{\hat{\mu}})$, $
\hat{r} \eqset \hat{r}_B + \hat{r}_{\F}$, $
\hat{\rho} \eqset 2\sqrt{\kappa(\o{F})} ~\hat{r}.$ 
Lemma~\ref{lemma:A_Frob_norm_new} combined with \eqref{def:corollary_bounds_F} yields
\begin{align*}
    |
    \bw(B_{\mu}, B_{\nu}) & - 
    \norm{\A\o{F}^{-1}\left(T_{\mu} - T_{\hat{\mu}} \right) }_{\F}
    |  \\
    &\le 
    \left(4r_B + 2\hat{r}_B \right) \norm{\A \left(B_{\mu} - B_{\nu}\right) }_{\F} + \norm{\A \Delta }_{\F}\\
    & =  
    \left(4r_B + 2\hat{r}_B \right) \norm{\A \left(\Delta + \o{F}^{-1} \left(T_{\mu} - T_{\hat{\mu}}\right)  \right) }_{\F} 
    + \norm{\A\Delta}_{\F} \\
    &\le \left(4r_B + 2\hat{r}_B \right) \norm{ \A\o{F}^{-1}\left(T_{\mu} - T_{\hat{\mu}} \right)}_{\F} + \left(1 + 4r_B + 2\hat{r}_B \right) \norm{\A\Delta}_{\F}\\
    &\overset{{ \text{by}~ (\ref{def:corollary_bounds_F}})}{\le} 
    \left(4r_B + 2\hat{r}_B + \hat{\rho}  \left(1 + 4r_B + 2\hat{r}_B\right) \right) \norm{\A}\norm{ \o{F}^{-1}\left(T_{\mu} - T_{\hat{\mu}} \right)}_{\F}\\
    &+\left( 1 + 4r_B + 2\hat{r}_B \right) \left(\rho + \hat{\rho}~\right)\norm{\A} \norm{ \o{F}^{-1} T_{\mu} }_{F} \\
    &\le 6 \kappa(\A)  (\hat{\rho} + \rho ) \norm{\A\o{F}^{-1} \left(T_{\mu} - T_{\hat{\mu}} \right) } + 4 (\hat{\rho} + \rho ) \norm{\A}
    \norm{\o{F}^{-1} T_{\mu}}.
\end{align*}
\end{proof}

The next result estimates the proximity of $B_{\mu}$ and $B$ in terms of $\norm{\Tmw}_{\F}$.
\begin{lemma}
\label{lemma:concentrationQ_w}
    For $X \in \H(d)$ we denote $\o{\xi}(X) \eqdef B^{1/2} \o{F} \left(B^{1/2} X B^{1/2}\right) B^{1/2},$ and set
    \begin{equation}
    \label{def:cb}
    c_B \eqdef \frac{4 \norm*{B}}{\lmin(\o{\xi})}.
    \end{equation}
    Assume that $r(\o{F}, \o{F}_{\mu}) \le \frac{1}{2}$ and
    $\norm*{\Tmw}_{\F} \le \frac{4}{3 c_B}$.
	Then
	\[
	\norm*{B^{-1/2}B_{\mu}B^{-1/2} - I}_{\F} \le c_B \norm*{\Tmw}_{\F}.
	\]
\end{lemma}

\begin{proof}
	First, we set
 	\[
	\o{\xi}_{\mu}(X) = B^{1/2} \o{F}_{\mu} \left(B^{1/2} X B^{1/2}\right) B^{1/2},
    \quad
    \zeta_{\mu} = \frac{1}{\lmin(\o{\xi}_{\mu})} \norm*{B^{1/2} \Tmw B^{1/2}}_{\F}.
	\]
    Provided that $\zeta_{\mu} \le \frac{2}{3}$, Lemma~B.1 by~\citet{kroshnin2019statistical} ensures
	\[
	\norm*{B^{-1/2} B_{\mu} B^{-1/2} - I}_{\F} \le \frac{\zeta_{\mu}}{1 - \tfrac{3}{4} \zeta_{\mu}} \le 2 \zeta_{\mu}. 
	\]
    Now we show that condition $\zeta_{\mu} \le \frac{2}{3}$ holds. 
    Assumption $r(\o{F}, \o{F}_{\mu}) \le \frac{1}{2}$ implies $r(\o{\xi}, \o{\xi}_{\mu})\le \frac{1}{2}$.
    This yields $\lmin(\o{\xi}_{\mu}) \ge \frac{\lmin(\o{\xi})}{2}$.
    Therefore, the assumptions of the lemma ensure $	\zeta_{\mu} \le \frac{c_B}{2} \norm*{\Tmw}_{\F} \le \frac{2}{3}.$
    This finishes the proof.
\end{proof}

In the rest of this section, we will use the following notations,
\[
r_B \eqset r(B, B_{\mu}),
\quad
r_{F} \eqset r(\o{F}, \o{F}_\mu),
\quad
r \eqset r_B + r_F,
\quad
\rho \eqset 2\sqrt{\kappa(\o{F})}r.
\]

The next lemma bounds the operator $\o{D}_{\mu}$ defined in \eqref{def:D_mu}. This result is crucial for the proof of Lemma~\ref{lemma:Qw_expansion}.

\begin{lemma}[Bounds on $\o{D}_{\mu}$]
\label{lemma:D_bounds} If $r \le \frac{1}{2}$, then
    \begin{equation*}
    	\frac{1}{1 + 2 r} \o{F}
    	\preccurlyeq \o{D}_{\mu}
    	\preccurlyeq \frac{1}{1 - r} \o{F}.
    \end{equation*}
\end{lemma}

\begin{proof}
    Let $B_t = (1 - t) B + t B_{\mu}$. Lemma~A.4 by~\citet{kroshnin2019statistical} ensures 
	\begin{equation*}
    	\frac{1}{1 - r_B} \dT{B}{S}
    	\preccurlyeq 
    	\int\limits^{1}_{0} \dT{B_{t}}{S} \diff t
    	\preccurlyeq
    	\frac{1}{1 + \tfrac{3}{4} r_B} \dT{B}{S}. 
	\end{equation*}
 
	Now recall the definition on the operator $\o{F}_{\mu}$ (see \eqref{def:T_mu}). Integrating the above inequality over $\diff \mu(S)$, we get
	\begin{equation*}
    	\frac{1}{1 + \tfrac{3}{4} r_B} \o{F}_{\mu}
    	\preccurlyeq \o{D}_{\mu}
    	\preccurlyeq \frac{1}{1 - r_B} \o{F}_{\mu}.
	\end{equation*}
    Since $r_F = \norm{\o{F}^{-1/2}\o{F}_{\mu}\o{F}^{-1/2} - I}$, it holds $   \left(1 - r_F\right) \o{F} 
        \preccurlyeq \o{F}_{\mu}
        \preccurlyeq \left(1 + r_F\right) \o{F}.$ Combining these bounds, we obtain:
    \begin{gather*}
        \frac{1}{1 + 2 r} \o{F}
        \preccurlyeq \frac{1 - r_F}{1 + \tfrac{3}{4} r_B} \o{F}
        \preccurlyeq \o{D}_{\mu}
        \preccurlyeq \frac{1 + r_F}{1 - r_B} \o{F}
        \preccurlyeq \frac{1}{1 - r} \o{F}.
    \end{gather*}
\end{proof}

\section{Gaussian approximation result (GAR)}
\label{sec:main_gauss_approx}

This section presents the general Gaussian approximation result. 
It is the key ingredient for bootstrap validity. The first lemma contains an auxiliary term $\gamma(\cdot)$. To avoid breaking the logic of the presentation, we will define $\gamma(\cdot)$ immediately after the lemma. Moreover, from now on, we will denote generic absolute constants as $\CONST$.

Now we define $\gamma(\cdot)$. Let $\o{K}$ be a positive semi-definite Hilbert--Schmidt operator. We assume its eigenvalues $\{\lambda_k\}_k$ are arranged in non-increasing order. We define 
\begin{equation}
    \label{def:kappa}
	\varkappa(\o{K}) \eqdef \left(\Lambda_1 \Lambda_2 \right)^{-1/2}~~\text{with}~~
	\Lambda^2_r \eqdef \sum_{k \ge r} \lambda^2_k,~\text{where}~r = 1, 2.
\end{equation}
Lemma~\ref{lemma:bounds_kappa} investigates the properties $\varkappa(\o{K})$. Let
\begin{equation}
    \label{def:gamma0}
    \gamma(\o{K}) \eqdef \varkappa(\o{K}) \tr(\o{K}).
\end{equation}

Note that the function $\gamma(\o{K})$ is dimension-free (i.e., scale-invariant). Moreover, $\gamma(\o{K}) \ge 1$.
This follows from the fact that for any $r \ge 1$ it holds 
\[
\Lambda^2_r \le \left(\sum_{k \ge r} \lambda_k\right)^2 \le \left(\tr(\o{K})\right)^2.
\]

\begin{proof}[Proof of Lemma~\ref{lemma:GAR_general}]
    The union bound ensures
    \begin{align*}
        &\P\left\{X \le z \right\} 
        \le \P\left\{Y \le \tfrac{z + m}{1 - \rho} \right\} + \P\left\{\abs{X - Y} > \rho Y + m\right\}
        \le \P\left\{Y \le \tfrac{z + m}{1 - \rho} \right\} + \delta, \\
        &\P\left\{Y \le \tfrac{z - m}{1 + \rho} \right\}
        \le \P\left\{X \le z \right\} + \P\left\{\abs{X - Y} > \rho Y + m\right\}
        \le \P\left\{X \le z \right\} + \delta.
    \end{align*}
    
    Thus
    \begin{equation*}
        \P\left\{Y \le \tfrac{z - m}{1 + \rho} \right\} - \delta 
        \le \P\left\{X \le z \right\} 
        \le \P\left\{Y \le \tfrac{z + m}{1 - \rho} \right\} + \delta.
    \end{equation*}
    
    Assumption \eqref{eq:gar_aux} yields
    \begin{equation*}
        \P\left\{\norm{G}_H \le \tfrac{z - m}{1 + \rho} \right\} - \delta - \Delta 
        \le \P\left\{X \le z \right\} 
        \le \P\left\{\norm{G}_H \le \tfrac{z + m}{1 - \rho} \right\} + \delta + \Delta.
    \end{equation*}
    
    Now one has to bound $\P\left\{\norm{G}_H \le \tfrac{z - m}{1 + \rho} \right\}$ and $\P\left\{\norm{G}_H \le \tfrac{z + m}{1 - \rho} \right\}$. The assumption of the lemma $\rho \in \left[0, \tfrac{1}{2}\right]$ together with Lemma~\ref{lemma:anticoncentration} yield
    \begin{align*}
        \P\left\{\norm{G}_H \le \tfrac{z - m}{1 + \rho} \right\} & \ge \P\left\{\norm{G}_H \le \tfrac{z}{1 + \rho} \right\} - \CONST \gamma(\o{K}) \frac{m}{\sqrt{\tr(\o{K})}} ,\\
        \P\left\{\norm{G}_H \le \tfrac{z + m}{1 - \rho} \right\} & \le \P\left\{\norm{G}_H \le \tfrac{z}{1 - \rho} \right\} + \CONST \gamma(\o{K}) \frac{m}{\sqrt{\tr(\o{K})}} .
    \end{align*}
    
    Now we consider a Gaussian r.v.\ $\alpha G$ with some $\alpha > 0$.
    Note that by definition $\varkappa(\alpha^2 \o{K}) = \frac{1}{\alpha^2} \varkappa(\o{K})$.
    To compare $G$ and $\alpha G$ we use Corollary~2.3 by~\citet{gotze2017gaussian}. 
    This ensures for any $z > 0$
    \begin{align*}
        \abs*{\P\left\{\norm{G}_H \le \tfrac{z}{\alpha}\right\} -  \P\left\{\norm{G}_H \le z\right\}} 
        &\le \CONST \left(\varkappa(\o{K}) + \varkappa(\alpha^2 \o{K})\right) \norm*{\o{K} - \alpha^2 \o{K}}_{1}\\
        &= \CONST \left(1 + \tfrac{1}{\alpha^2}\right) \abs{1 - \alpha^2} \varkappa(\o{K}) \tr(\o{K}).
    \end{align*}
    
    Setting $\alpha = {1 + \rho}$ and taking into account that $\rho \in [0, \tfrac{1}{2}]$, we obtain
    \[
    \P\left\{\norm{G}_H \le \tfrac{z}{1 + \rho}\right\} 
    \ge \P\left\{\norm{G}_H \le z\right\} - \CONST \gamma(\o{K}) \rho.
    \]
    In a similar way,
    \[
    \P\left\{\norm{G}_H \le \tfrac{z}{1 - \rho}\right\} 
    \le \P\left\{\norm{G}_H \le z\right\} + \CONST \gamma(\o{K}) \rho.
    \]
    Collecting all the bounds, we get the result.
\end{proof}

The next lemma investigates
the properties of $\varkappa(\cdot)$ introduced by \eqref{def:kappa}.

\begin{lemma}[Bounds on $\varkappa(\cdot)$]
\label{lemma:bounds_kappa}
	Let $\o{\Psi}$ and $\o{\Phi}$ be symmetric operators, such that
	$
	\norm*{\o{\Phi} - \o{\Psi}}_1
	\le \frac{\Lambda_2^2\left(\o{\Psi}\right)}{4 \norm{\o{\Psi}}},
	$
    with $\norm{\cdot}_1$ be $1$-Schatten norm. Then the following bounds hold,
	\[
	\varkappa(\o{\Phi}) \le 2 \varkappa(\o{\Psi}),
	\quad
	\tr\o{\Phi} \le \tfrac{5}{4} \tr \o{\Psi}.
	\]
\end{lemma}

\begin{proof}
	Note, that
	$\Lambda_2^2\left(\o{\Psi}\right) \le \Lambda_1^2\left(\o{\Psi}\right) \le \norm*{\o{\Psi}} \tr(\o{\Psi})$
	and therefore,
	\[
    	\tr\left(\o{\Phi}\right) 
    	\le \tr(\o{\Psi}) + \norm*{\o{\Phi} - \o{\Psi}}_1
    	\le \frac{5}{4} \tr(\o{\Psi}).
	\]
	By the definition of $\Lambda^2_r(\cdot)$,
	$\Lambda_r^2\left(\o{\Phi}\right) 
	\ge \Lambda_r^2\left(\o{\Psi}\right) - \norm*{\o{\Psi}} \norm*{\o{\Phi} - \o{\Psi}}_1$,
	then
	\[
    	\Lambda_1^2\left(\o{\Phi}\right) \Lambda_2^2\left(\o{\Phi}\right)
    	\ge \Lambda_1^2\left(\o{\Psi}\right) \Lambda_2^2\left(\o{\Psi}\right) - \left(\Lambda_1^2\left(\o{\Psi}\right) 
    	+ \Lambda_2^2\left(\o{\Psi}\right)\right) \norm*{\o{\Psi}}\norm*{\o{\Phi} - \o{\Psi}}_1.
	\]
	Then it follows that
	\begin{align*}
    	\varkappa(\o{\Phi})
    	&\le \varkappa(\o{\Psi}) \left(1 - \tfrac{\Lambda_1^2\left(\o{\Psi}\right) + \Lambda_2^2\left(\o{\Psi}\right)}{\Lambda_1^2\left(\o{\Psi}\right) \Lambda_2^2\left(\o{\Psi}\right)} \norm*{\o{\Psi}} \norm*{\o{\Phi} - \o{\Psi}}_1\right)^{-1} \\
        &\le \varkappa(\o{\Psi}) \left(1 - 2 \tfrac{\norm*{\o{\Psi}}}{\Lambda_2^2\left(\o{\Psi}\right)} \norm*{\o{\Phi} - \o{\Psi}}_1\right)^{-1}
    	\le 2 \varkappa(\o{\Psi}). \qedhere
	\end{align*}
\end{proof}

\begin{lemma}[Anti-concentration]
\label{lemma:anticoncentration}
    Let $G \sim \ND(0, \o{K})$ be a Gaussian vector taking values in 
    some Hilbert space $H$. Then for any $\eps, x \ge 0$ the following
    anti-concentration bound holds:
	\[
	\P\left\{x \le \norm{G}_H \le x + \eps\right\} 
	\le \CONST \gamma(\o{K}) \frac{\eps}{\sqrt{\tr(\o{K})}}.
	\]
\end{lemma}

\begin{proof}
    For any $x, h, \eps > 0$ it holds that
    \[
    (x + \eps)^2 \le \begin{cases}
        x^2 \left(1 + \frac{\eps}{h}\right)^2, & h \le x, \\
        x^2 + 2 h \eps + \eps^2, & h > x .
    \end{cases}
    \]
    Thus, the union bound and Theorem~2.7 by~\citet{gotze2017gaussian} yield
	\begin{align*}
		\P\left\{x \le \norm{G}_H \le x + \eps\right\}
		\le & \P\left\{x^2 \le \norm{G}_H^2 \le x^2 + 2 h \eps + \eps^2\right\} + \P\left\{x \le \norm{G}_H \le x \left(1 + \frac{\eps}{h}\right) \right\} \\
		\le & \CONST \varkappa(\o{K}) \left(h \eps + \eps^2 + \frac{\eps}{h} \tr(\o{K})\right) 
		\le \CONST \varkappa(\o{K}) \left(\eps \sqrt{\tr(\o{K})} + \eps^2\right),
	\end{align*}    
	where the last inequality is ensured by $h = \sqrt{\tr\o{K}}$. 
	The above inequality can be rewritten as
	\[
	\P\left\{x \le \norm{G}_H \le x + \eps\right\} \le \CONST \gamma(\o{K}) \left(\frac{\eps}{\sqrt{\tr(\o{K})}} + \frac{\eps^2}{\tr(\o{K})}\right).
	\]
	Since $\gamma(\o{K}) \ge 1$ and the probability on the l.h.s.\ is bounded by $1$, it is enough to consider the case $\eps \le \sqrt{\tr(\o{K})}$. Thus, we obtain
	\[
	\P\left\{x \le \norm{G}_H \le x + \eps\right\} 
	\le \CONST \gamma(\o{K}) \frac{\eps}{\sqrt{\tr(\o{K})}}.
	\]
\end{proof}

\section{Generalized bootstrap in the Bures--Wasserstein space}
\label{sec:boot_validity}

Before we state the generalized bootstrap problem, we will specify framework presented in Section~\ref{seq:intro_gen_boot} for Bures--Wasserstein barycenters. Recall that the barycenter map
$\mathscr{B}$ \eqref{def:bar_map} is uniquely defined (Theorem~2.1 in~\citep{kroshnin2019statistical}).
Thus, according to Corollary~5 in \citep{le2017existence} it is continuous w.r.t.\ the $2$-Wasserstein metric on the subspace of probability measures $\mathcal{P}\bigl(\H_{++}(d)\bigr) \subset \mathcal{M}_{2}\bigl(\H_{++}(d)\bigr)$. Hence, by homogeneity $\mathscr{B}(\mu) = \mathscr{B}\left(\frac{\mu}{\mu(\H_{++}(d))}\right)$. Consequently, $\mathscr{B}$ is measurable on $\mathcal{M}\bigl(\H_{++}(d)\bigr)$. 

Let $(\Omega, \mathscr{F}, \mathbb{P})$ be a probability space and $\mu \in \mathcal{M}_{2}\bigl(\H_{++}(d)\bigr)$ be a random measure with distribution $\mathscr{P}$. Let $B \eqdef \mathscr{B}(\E \mu)$ and set $B_{\mu} = \mathscr{B}(\mu)$. The generalized bootstrap aims to
select a suitable random measure $\hat{\mu}$ depending on $\mu \in \mathcal{M}_{2}\bigl(\H_{++}(d)\bigr)$, such that $\bw(B, B_{\mu}) \overset{\mathrm{d}}{\approx} \bw(B_{\mu}, B_{\hat{\mu}})$. To prove the validity of this approximation, we will follow the framework presented in Section~\ref{seq:intro_gen_boot}. The first step is the linearization on the Bures--Wasserstein distance.

The key role in linearization of $\bw(B, B_\mu)$ will play $T_{\mu} = \mathscr{T}(\mu)$ and $\o{F}_{\mu} = \mathscr{F}(\mu)$ introduced by \eqref{def:T_mu}. Note that $\mu$ is random, consequently $T_\mu$ and $\o{F}_{\mu}$ are random as well. By construction, $\E T_{\mu} = 0 $. Let $\o{F}\eqdef -\E\o{F}_{\mu} $. The following assumptions are crucial for control of the tail behaviour of $T_{\mu}$ and $\o{F}_{\mu}$. We assume there exist functions $\eps_{T}(\x) > 0$ and $\eps_{\F}(\x) > 0$, such that,
\begin{align}
    \label{asm:T}
    &\P\left\{\norm*{\Tmw}_{\F} > \eps_{T}(\x) \right\} \le \CONST e^{-\x},  \tag{$T$}\\
    \label{asm:F}
    &\P\left\{r(\o{F}, \o{F}_{\mu}) > \eps_{\F}(\x)\right\} \le \CONST e^{-\x}
    \tag{$F$},
\end{align}
where $r(\o{F}, \o{F}_{\mu}) \eqdef \norm{\o{F}^{-1/2}\o{F}_{\mu} \o{F}^{-1/2} - \o{I}}$. 
Next, to ensure the Gassuan approximation of $\sqrt{n}\bw(B, B_n)$, we assume that there exists centred Gaussian vector $Z \sim \ND(0, \o{\Upxi})$ such that,
\begin{equation}
    \label{asm:GAR}
    \sup_{z>0}\abs*{\P\left\{ \norm*{\o{F}^{-1} \Tmw}_{\F} \le z \right\} 
    - \P\left\{\norm*{Z}_{\F} \le z \right\}} \le \eps_{G}. \tag{$G$}
\end{equation}

Assumptions \eqref{asm:T}, \eqref{asm:F}, and \eqref{asm:GAR} combined together ensure that $\bw(B_\mu, B) \overset{\mathrm{d}}{\approx} \norm{\A Z }_\F$, where
$ \A \eqdef \left(- \frac{1}{2} \dT{B}{B}\right)^{1/2}$. Before writing this result, we get some trivial but useful bounds. 

\begin{lemma}
\label{lemma:aux_bounds}
    Let Assumptions~\eqref{asm:T} and~\eqref{asm:F} be fulfilled. Then with probability at least $1 - \CONST e^{-\x}$ it holds
    \begin{gather}
        \label{def:aux_bounds_mu}
        r(B, B_{\mu}) \leq c_B \eps_{T}(\x) \quad 
        r(\o{F}, \o{F}_{\mu}) \leq \eps_{F}(\x), \quad 
        \rho \leq \eps(\x).
    \end{gather}

    If Assumptions~\eqref{asm:T_boot} and \eqref{asm:F_boot} hold as well, then, conditioned on any $\mu \in \mathscr{A}_{t}$, it holds with probability $1 - \CONST e^{-\x}$,
    \begin{gather}
    \label{def:aux_bounds_nu}
        r(B, B_{\hat{\mu}}) \leq c_B \hat{\eps}_{T}(\x, \t), 
        \quad 
        r(\o{F}, \o{F}_{\hat{\mu}}) \leq \hat{\eps}_{\F}(\x, \t),
        \quad
        \hat{\rho} \leq \hat{\eps}(\x, \t).
    \end{gather}
\end{lemma}

\begin{proof}
    The proof trivially follows from Lemma~\ref{lemma:concentrationQ_w}.
\end{proof}

\begin{lemma}[Gaussian approximation for $\bw(B_{\mu}, B)$]
\label{lemma:gar_boot}
    Let Assumptions~\eqref{asm:T},~\eqref{asm:F}, and~\eqref{asm:GAR} be fulfilled. Denote $\A \eqdef \left(- \frac{1}{2} \dT{B}{B}\right)^{1/2}$ and set $\oUpxi' \eqdef \A\oUpxi \A$.
    Then
    \begin{gather*}
        \sup_{z \ge 0}\abs*{\P\left\{ \bw(B_{\mu}, B) \le z\right\} - \P\left\{\norm{\A Z}_{\F} \le z\right\}}
    	\le \mathcal{E}, \nonumber \\
    	\label{def:Omega_BW_w}
        \mathcal{E}  
        \eqdef \eps_{G} + \CONST \cdot \inf_{\x \in \mathscr{X}} \left\{e^{-\x} + \sqrt{\kappa(B)}\gamma(\oUpxi') \eps(\x) 
        \right\}, \quad
        \mathscr{X} \eqdef \left\{\x :~ \eps(\x) \le \frac{1}{6 \sqrt{ \kappa(B)}}  \right\}.
    \end{gather*}
\end{lemma}

\begin{proof}
    We set $X = \bw(B_{\mu}, B)$, 
    $Y = \norm*{\A \o{F}^{-1} \Tmw}_{\F}$, 
    $G = \A Z$. 
    Assumption~\eqref{eq:comp_aux} holds due to \eqref{bound:B_OT_BW} and \eqref{def:aux_bounds_mu}:
    \[
    \abs{X - Y} \le 3 \sqrt{\kappa(B)} \rho {Y} 
    \le 3 \sqrt{\kappa(B)} \eps(\x) Y,
    \]
    the last inequality holds with 
    probability at least $1 - \CONST e^{-\x}$ for all $\x $, such that $\eps(\x) \le \frac{1}{6 \sqrt{\kappa(B)}}$.
    
    Assumption (GAR-II) is fulfilled due to Assumption~\eqref{asm:GAR}.
    The result follows immediately from Lemma~\ref{lemma:GAR_general}. 
\end{proof}

Next, we are to ensure a similar result for the measure $\hat{\mu}$.
Let $T_{\hat{\mu}} \eqdef T_{\hat{\mu}}(B_\mu)$ and $\o{F}_{\hat{\mu}} \eqdef \o{F}_{\hat{\mu}}(B)$. Recall that $\hat{\mu}$ might depend on $\mu$. We assume that there exists a Borel set $\mathscr{A}_{t} \subset \mathcal{M}(\H_{++}(d))$, such that $\P\{\mu \in \mathscr{A}_{t} \} \ge 1 - \CONST e^{-t}$). The following assumptions hold on this event. We assume there exist functions $\hat{\eps}_{T}(\x,\t) > 0$ and $\hat{\eps}_{F}(\x,\t) > 0$, such that, 
\begin{align}
    \label{asm:T_boot}
    &\P\left\{\norm*{T_{\hat{\mu}} - T_{\mu}}_{\F} > \hat{\eps}_{T}(\x,\t) ~|~\mu \right\} \le \CONST e^{-\x}, \tag{$\hat{T}$}\\
    \label{asm:F_boot}
    &\P\left\{r(\o{F}, \o{F}_{\hat{\mu}}) > \hat{\eps}_{F}(\x,\t) ~|~\mu\right\} \le \CONST e^{-\x}
    \tag{$\hat{F}$},
\end{align}
where $r(\o{F}, \o{F}_{\hat{\mu}}) \eqdef \norm{\o{F}^{-1/2}\o{F}_{\hat{\mu}} \o{F}^{-1/2} - \o{I}}$. 
Finally, let $Z_{\mu} \sim \ND(0, \o{\Upxi}_{\mu})$ be centred Gaussian matrix such that,
\begin{equation}
\label{asm:GAR_boot}
    \sup_{z>0}\abs*{\P\left\{\norm*{ \o{F}^{-1} \left( T_{\hat{\mu}} - \Tmw \right)}_{\F} \le z ~|~\mu \right\} 
    - \P\left\{\norm*{Z_{\mu}}_{\F} \le z ~|~\mu\right\}} \le \hat{\eps}_{G}(\t). \tag{$\hat{G}$}
\end{equation}
Assumptions \eqref{asm:T_boot}, \eqref{asm:F_boot}, and \eqref{asm:GAR_boot} ensure $\bw(B_{\hat{\mu}}, B_\mu) \overset{\mathrm{d}}{\approx} \norm{\A Z_\mu }_\F$.

\begin{lemma}[Gaussian approximation for $\bw(B_{\mu}, B_{\hat{\mu}})$]
\label{lemma:gar_boot_BW}
    Denote $\oUpxi'_\mu \eqdef \A\o{\Upxi}_{\mu}\A$. Let Assumptions~\eqref{asm:T_boot},~\eqref{asm:F_boot}, and~\eqref{asm:GAR_boot} be fulfilled. Then, on the event \\
    $\left\{\mu \in \mathscr{A}_t, \;\rho \le \frac{1}{12\sqrt{\kappa(B)}}\right\}$
    it holds that
    \begin{gather}
        \sup_{z \ge 0} \abs*{
        \P\left\{
        \bw(B_{\hat{\mu}}, B_{\mu}) \le z ~\middle|~\mu\right\} - \P\left\{\norm{\A Z_{\mu}}_{\F} \le z ~\middle|~\mu
        \right\}}
        \le \hat{\mathcal{E}}(\t),\nonumber \\ 
        \hat{\mathcal{E}}(\t) 
        \eqdef \hat{\eps}_{G}(\t) + \CONST \cdot \inf_{\x \in \hat{\mathscr{X}}(\t)} \left\{ e^{-\x} + \gamma(\o{\Upxi}'_{\mu})\sqrt{\kappa(B)} \left(\rho + \hat{\eps}(\x, \t)\right)
        \biggl( \frac{\norm{\A}\norm{ \o{F}^{-1} T_{\mu}}_{\F}}{\sqrt{\tr(\oUpxi'_\mu)}} + 1 \biggr) \right\}, \label{def:Omega_BW_boot} 
    \end{gather}
    where $\hat{\mathscr{X}}(\t) \eqdef \left\{\x:~ \hat{\eps}(\x, \t) \le \frac{1}{12 \sqrt{ \kappa(B)}}\right\}.$
\end{lemma}

\begin{proof}
    We have to check Assumptions~\eqref{eq:comp_aux} and \eqref{eq:gar_aux}. We set
    $X = \bw(B_{\hat{\mu}}, B_{\mu})$, $Y = \norm*{\A\o{F}^{-1}\left( T_{\hat{\mu}} -\Tmw\right)}_{\F}$, $G = \A Z_{\mu}$. Assumption~\eqref{eq:comp_aux} is valid due to Corollary~\ref{def:corollary_bounds_F_BW} and assumptions $r\le \frac{1}{2}, \hat{r}\le \frac{1}{2}$,
    \begin{align*}
        |X - Y| \le 6\kappa(\A) \left(\hat{\rho} + \rho\right)Y 
        +  4 \left( \hat{\rho} + \rho \right)\norm{\A}  \norm{ \o{F}^{-1} T_{\mu} }_{\F}.
    \end{align*}

    Note that Lemma~\ref{lemma:spec_A} ensures $\kappa(\A) = \sqrt{\kappa(B)}$.
    Using Lemma~\ref{lemma:aux_bounds}, we get 
    \begin{align*}
        |X - Y| \le 6 \sqrt{\kappa(B)} \left(\rho + \hat{\eps}(\x, \t)\right) Y +  4  \left(\rho + \hat{\eps}(\x, \t)\right)\norm{\A} \norm{\o{F}^{-1} T_{\mu}}_{\F}.
    \end{align*}
    The inequality holds with probability at least $1 - \CONST e^{-\x}$ for all $\x $ such that $\hat{\eps}(\x, \t) + \rho \le \frac{1}{6\sqrt{\kappa(B)}}$. Since by assumption of the lemma $\rho \le \frac{1}{12 \sqrt{\kappa(B)}}$, we get $\hat{\eps}(\x, \t) \le \frac{1}{12 \sqrt{\kappa(B)}}$.

    Assumption (GAR-II) is valid due to Assumption~\eqref{asm:GAR_boot} with $\Delta = \hat{\eps}_{G}(\t)$. 
    The claim follows.
\end{proof}

To complete the proof of generalized bootstrap validity, i.e., to show that $\bw(B, B_{\mu})\overset{\mathrm{d}}{\approx} \bw(B_{\hat{\mu}}, B_\mu)$, we assume that
\begin{equation}
    \label{asm:Sigma}
    \P\left\{\norm*{\oUpxi - \oUpxi_{\mu}}_1 > \eps_{\Upxi}(\x) \right\} \le \CONST e^{-\x}  \tag{$\Upxi$},
\end{equation}
with $\norm{\cdot}_{1}$ being $1$-Schatten norm.
This assumption entails
$\norm*{\A Z}_{\F}
\overset{\mathrm{d}}{\approx}
\norm*{\A Z_{\mu}}_{\F}$. Now we are ready to present the main result. Let us denote
\begin{equation*}
    \eps(\x) \eqdef 6 \sqrt{\kappa(\o{F})} \left( 
    c_B \eps_{T}(\x) + \eps_{\F}(\x) \right),
    \quad
    \hat{\eps}(\x, \t) \eqdef 6 \sqrt{\kappa(\o{F})} \left( c_B\hat{\eps}_{T}(\x, \t) + \hat{\eps}_{F}(\x, \t)  \right),
\end{equation*}
with $c_B$ coming from \eqref{def:cb}.

\begin{theorem}[Bootstrap validity]
\label{theorem:boot_BW}
    Let all Assumptions~\eqref{asm:T}--\eqref{asm:Sigma} be fulfilled. Denote $\oUpxi' \eqdef \A\oUpxi\A$ and let $\t \ge 0$ be such that
    \begin{equation}
    \label{eq:bound_upxi_bw}
        \eps_{\Upxi}(\t)
        \le \CONST \frac{\Lambda_2^2\left(\oUpxi'\right)}{\norm{\A}^2 \norm*{\oUpxi'}}.
    \end{equation}
    Then with probability at least $1 - \CONST e^{-\t}$,
    \begin{gather*}
        \sup_{z\ge 0} \abs*{\P\left\{\bw(B_{\mu},B) \le z\right\}
        - \P\left\{ \bw(B_{\hat{\mu}},B_{\mu}) \le z ~\middle|~\mu\right\}} \leq \Gamma(\t),
	\end{gather*}
    \begin{align*}
        \Gamma(\t) & =  \CONST \varkappa(\oUpxi')\norm{\A}^2 
        \eps_{\Upxi}(\t) + \eps_{G} + \CONST \cdot \inf_{\x \in \mathscr{X}} \left\{e^{-\x} + \gamma(\oUpxi')\sqrt{\kappa(B)} \eps(\x) \right\} + \hat{\eps}_{G}(\t) \\
        &\quad + \CONST \cdot\inf_{\x \in \hat{\mathscr{X}}(\t)}\left\{ e^{-\x} + \gamma(\oUpxi')\sqrt{\kappa(B)}  \left(\eps(\t) + \hat{\eps}(\x, \t)\right)
        \left( \tfrac{\norm{\A}  \norm{ \o{F}^{-1}}}{\sqrt{\tr(\oUpxi')}}\eps_{T}(\t) + 1 \right)\right\},
    \end{align*}
    where $\mathscr{X} \eqdef \left\{\x :~ \eps(\x) \le \frac{1}{6 \sqrt{ \kappa(B)}}  \right\}$, and 
    $
    \hat{\mathscr{X}}(\t) \eqdef \left\{\x:~ \hat{\eps}(\x, \t) \le \frac{1}{12 \sqrt{ \kappa(B)}}\right\}.
    $
\end{theorem}

\begin{proof}
    Lemma~\ref{lemma:gar_boot} and~\ref{lemma:gar_boot_BW} ensure that
    for all $z \ge 0$ with probability at least $1 - \CONST e^{-\t}$, it holds
	\begin{gather*}
    	\abs*{\P\left\{\bw(B_{\mu}, B) \le z\right\} - \P\left\{\norm*{\A Z}_{\F} \le z\right\}}
    	\le \mathcal{E},\\
    	\abs*{\P\left\{\bw(B_{\hat{\mu}}, B_{\mu}) \le z~|~\mu \right\} - \P\left\{\norm*{\A Z_{\mu}}_{\F} \le z~|~\mu\right\}}
    	\le \hat{\mathcal{E}}(\t).
	\end{gather*}
	This yields
	\begin{align}\label{eq:3} 
        &\abs*{\P\left\{\bw(B_{\mu},B) \le z\right\}
        - \P\left\{\bw(B_{\hat{\mu}},B_{\mu}) \le z \;\middle|\;\mu\right\}} \\
        &\le \abs*{\P\left\{\norm*{\A Z}_{\F} \le z\right\} - \P\left\{\norm*{\A Z_{\mu}}_{\F} \le z \;\middle|\; \mu\right\}} + \mathcal{E} + \hat{\mathcal{E}}(\t).
	\end{align}

    First, we consider $\hat{\mathcal{E}}(\t)$ coming from Lemma~\ref{lemma:gar_boot_BW}:
    \begin{align*}
        \hat{\mathcal{E}}(\t) = \hat{\eps}_{G}(\t) + \CONST \cdot \inf_{\x \in \hat{\mathscr{X}}(\t)} \left \{e^{-\x} + \sqrt{\kappa(B)}\gamma(\o{\Upxi}'_{\mu}) \left(\rho + \hat{\eps}(\x, \t)\right)
        \left( \frac{\norm{\A}\norm{ \o{F}^{-1} T_{\mu}}_{\F}}{\sqrt{\tr(\o{\Upxi}'_{\mu})}} + 1 \right) \right\}.
    \end{align*}
    Lemma~\ref{lemma:aux_bounds} and Assumption~\eqref{asm:T}
    ensure with probability at least $1-\CONST e^{-\x'}$, that $\rho \le \eps(\x')$,\\
    $\norm{ \o{F}^{-1} T_{\mu}}_{\F} \le \norm{ \o{F}^{-1}} \eps_{T}(\x')$. Further, condition \eqref{eq:bound_upxi_bw} and Lemma~\ref{lemma:bounds_kappa} ensure
    \begin{equation}
    \label{eq:yet_another_name}
        \varkappa( \oUpxi'_{\mu}) \le 2 \varkappa\left(\oUpxi' \right),
        \quad 
        \tr (\oUpxi'_{\mu}) \le \frac{5}{4} \tr\left(\oUpxi' \right).     
    \end{equation}
    
    Taking into account the definition of $\gamma(\cdot)$ \eqref{def:kappa}, we get that with probability at least $1-e^{-\x'}$,
    \begin{align*}
        \hat{\mathcal{E}}(\t) \le \hat{\eps}_{G}(\t) + \CONST \inf_{\x \in \hat{\mathscr{X}}(\t)} \left \{e^{-\x} + \sqrt{\kappa(B)}\gamma(\o{\Upxi}')  \left(\eps(\x') + \hat{\eps}(\x, \t)\right)
        \left( \frac{\norm{\A}\norm{ \o{F}^{-1} }}{\sqrt{\tr(\o{\Upxi}')}}\eps_{T}(\x') + 1 \right) \right\}.
    \end{align*}

    Next, we have to bound 
    $\abs*{\P\left\{\norm*{\A Z}_{\F} \le z\right\} - \P\left\{\norm*{\A Z_{\mu}}_{\F} \le z ~|~\mu\right\}}.$ 
    Recall that $\A$ is self-adjoint. Corollary~2.3 by~\citet{gotze2017gaussian} ensures,
	\begin{align}
	    \label{eq:4}
    	\sup_{z\ge 0}\abs*{\P\left\{\norm*{\A Z}_{\F} \le z\right\} - \P\left\{\norm*{\A Z_{\mu}}_{\F} \le z~|~\mu\right\}}  
        \le \CONST 
        \left(
        \varkappa\left(\oUpxi'\right)
        + \varkappa\left(\oUpxi'_{\mu}\right)
        \right) 
        \norm*{\oUpxi' - \oUpxi'_{\mu}}_1. 
	\end{align}

    Taking into account \eqref{eq:yet_another_name} and Assumption~\eqref{asm:Sigma}, we get with probability at least $1-\CONST e^{-\y}$
    \[
    \norm*{\oUpxi' - \oUpxi'_{\mu}}_1
    \le \norm*{\A}^2 \norm*{\oUpxi - \oUpxi_{\mu}}_1 \le \norm*{\A}^2 \eps_{\Upxi}(\y).
    \]
    
    Combining these bounds with \eqref{eq:3} and \eqref{eq:4} and setting $\y = \x' = \t$, we get the result.
\end{proof}

\section{Multiplier bootstrap in the Bures--Wasserstein space}
\label{sec:proof_sub_gauss}

\begin{proof}[Proof of Lemma~\ref{lemma:aux_behaviour}]
    First, Assumption~\eqref{asm:subexp} ensures,
    \begin{equation}
        \label{eq:useful_one}
        \norm*{\norm*{S}^{1/2}}_{\psi_2} \le \norm*{\sqrt{\tr S}}_{\psi_2} < +\infty.
    \end{equation}

    Now we recall that $T_{B}^{S} = B^{-1/2} \left(B^{1/2} S B^{1/2}\right)^{1/2} B^{-1/2}.$ Using \eqref{eq:useful_one}, we get\\ $ \norm*{T_{B}^{S}} \le \frac{\lmax^{1/2}(B)}{\lmin(B)} \norm*{S}^{1/2}.$ Thus, $\norm{\norm{T_{B}^{S}}_{\F}}_{\psi_{2}} \le d\cdot v_S. $
    Finally, we use the result (III) in Lemma~\ref{lemma:dT} that ensures 
    \begin{equation*}
        \label{eq:bound_on_dT}
        \norm*{\dT{B}{S}} \le \frac{\lmax^{1/2}\left(S^{1/2} B S^{1/2}\right)}{2 \lmin^2(B)} \le \frac{\lmax^{1/2}(B)}{2 \lmin^2(B)} \norm*{S}^{1/2}.
    \end{equation*}
    Combining this fact with~\eqref{eq:useful_one}, we get the result.   
\end{proof}

Before validating the bootstrap assumptions, we prove two auxiliary lemmas. The first lemma deals with concentration of sub-exponential random variables. The first two results are well-known and we provide them for the sake of completeness.

In the following, we will often use the auxiliary concentration results. For the sake of completeness, we provide them below. Let $(x)_{+} = \max\{0, x\}$ and $\log(x) = \max\{1, \ln x\}$.

\begin{statement}[Theorem~2.1 \citep{kroshnin2024bernstein}]
\label{stat:iid_concentration}
    Fix $\alpha \ge 1$. Let $\o{X}_1, \dots, \o{X}_n \in \H(d)$ be independent and $\E \o{X}_i = 0$ for all $i$. Define
    \begin{gather*}
    K \eqdef \max_{i} \norm*{ \norm{\o{X}_i } }_{\psi_{\alpha}}, 
    \quad
    U^2 \eqdef \sumi \norm*{ \norm{\o{X}_i } }^2_{\psi_{\alpha}},
    \quad
    \sigma^2 \eqdef \norm*{ \sumi \E \o{X}^2_i},
    \quad
     z \eqdef \left(\log \frac{U}{\sigma}\right)^{1/\alpha} .
    \end{gather*}
    Then, with probability at least $1 - 2 d e^{-\x}$,
    \[
    \norm*{\frac{1}{n} \sumi \mathbf{X}_i} 
    \lesssim \sigma \sqrt{\frac{\x}{n}} + K z \frac{\x}{n}.
    \]
\end{statement}

\begin{statement}[Corollary~3.5 \citep{kroshnin2024bernstein}]
\label{stat:norm_concentration}
    Fix $\alpha \ge 1$. Let $(\mathcal{H}, \norm{\cdot}_{H})$ be a separable Hilbert space and assume $X_1, \dots, X_n \in \mathcal{H}$ are independent random variables such that $\E X_i = 0$. Define
    \begin{gather*}
    K \eqdef \max_{i} \norm*{\norm*{X_i}_{H}}_{\psi_{\alpha}},
    \quad
    U^2 \eqdef \sumi \norm*{\norm*{X_i}_{H}}^2_{\psi_{\alpha}},
    \quad
    \sigma^2 \eqdef \sumi \E \norm*{X_i}^2_{H}, 
    \quad
    z \eqdef \left(\log \frac{U}{\sigma}\right)^{1/\alpha} .
    \end{gather*}
    Then for $\x\ge 1$ with probability at least $1-e^{-\x}$ 
    \[
    \norm*{\frac{1}{n}\sumi X_i}_{H} \lesssim  \sigma \sqrt{\frac{\x}{n}} +  K z \frac{\x}{n}.
    \]
\end{statement}

\begin{statement}
    \label{lemma:subexp_bounds}
    Fix $\alpha > 0$. Let $X_1, \dots, X_n \ge 0$ be i.i.d.\ random variables, such that $\sigma^2 = \E X^2_1$, $v = \norm*{X_1}_{\psi_{\alpha}}$. Let $z \eqdef \left(\log \frac{v}{\sigma}\right)^{1/\alpha}$. Then for any $p \ge 2$ and $\x \ge 0$ it holds with probability at least $1-2e^{-\x}$
    \[
    \frac{1}{n}\sumi X^{p}_i \lesssim \sigma^2( v z)^{p-2} + v^p\left(z^p + (\x + \log n)^{\frac{p}{\alpha} - 1} \right)\frac{\x}{n}.
    \]
    Moreover, $\max_i X_i \le v(\x + \ln 2n)^{1/\alpha}.$
\end{statement}

\begin{proof}
    Theorem~2.1 from \citep{kroshnin2024bernstein} ensures that
    \begin{align*}
        \frac{1}{n}\sumi X^{p}_i & \lesssim \E X^p_i + \sqrt{\frac{\x}{n}\E X^{2p}_i} + v^p\left( \log \frac{v^{2p}}{\E X^{2p}_i}\right)^{\frac{p}{\alpha}}\frac{\x}{n}
        + v^p(\x + \ln n)^{\frac{p}{\alpha} -1 }\frac{\x}{n}\\
        & \lesssim \sigma^2\left(v\left( \log \frac{v}{\sigma}\right)^{\frac{1}{\alpha}}  \right)^{p-2} 
        + \sigma\left(v\left( \log \frac{v}{\sigma}\right)^{\frac{1}{\alpha}}  \right)^{p-2}\sqrt{\frac{\x}{n}}\\
        &
        + v^p\left( \left( \log \frac{v}{\sigma}\right)^{\frac{p}{\alpha}} + (\x + \ln n)^{\frac{p}{\alpha} -1 } \right) \frac{\x}{n}\\
        & \lesssim \sigma^2\left(v z  \right)^{p-2}
        + 
        v^p\left( z^p + (\x + \ln n)^{\frac{p}{\alpha} -1 } \right) \frac{\x}{n}.
    \end{align*}

    To get the second result, we use a well-known line of reasoning,
    \[
        \P\left\{\max_i X_i \ge \t\right\} 
        = \P\left\{\bigcup_{i} \left\{X_i \ge \t\right\} \right\} 
        \le 2 n e^{-  (\t / v)^{\alpha}} = e^{\ln(2 n) - (\t / v)^{\alpha}}. \qedhere
    \]
\end{proof}

Throughout the rest of the text, we denote $T_i \eqdef T^{S_i}_{B} - I.$ We also write down explicitly all the terms.
The $\mathscr{T}$-mappings are written as
\begin{equation*}
    T_{\mu} \eqset \mathscr{T}(P_n) = \frac{1}{n} \sumi T_i,
    \quad 
    T_{\hat{\mu}} \eqset \mathscr{T}(P_w) =  \frac{1}{n} \sumi w_i T_i,
\end{equation*}
and the $\mathscr{F}$-mappings are
\begin{gather*}
\label{def:subg_F}
    \o{F} = \mathscr{F}(P) = -\E\dT{B}{S},
    \quad
    \o{F}_{\mu} = \mathscr{F}(P_n) = \frac{1}{n} \sumi \dT{B}{S_i},\\
    \o{F}_{\hat{\mu}} = \mathscr{F}(P_w) = \frac{1}{n} \sumi  w_i\dT{B}{S_i}.
\end{gather*}

The vectors used for Gaussian approximation are $Z\sim \ND(0, \oUpxi)$ and $Z_{\mu}\sim \ND(0, \oUpxi_{\mu})$, where 
\begin{gather*}
    \oUpxi \eqdef \frac{1}{n}\o{F}^{-1} \left[\E \left(T^{S}_{B} - I \right) \otimes \left(T^{S}_{B} - I \right) \right] \o{F}^{-1},\\
    \oUpxi_{\mu} \eqdef \frac{1}{n}\o{F}^{-1}\left[
         \frac{1}{n} \sumi \left(T^{S_i}_{B} - I \right) \otimes \left(T^{S_i}_{B} - I \right) \right] \o{F}^{-1},
\end{gather*}
with $\otimes$ denoting the tensor product. Throughout this section, we denote 
\[
\sigma^2_T \eqdef \E \left\|T_1 \right\|^2_{\F},
\quad
C_T \eqdef \frac{v^2_T}{\sigma^2_T} \log\frac{v_T}{\sigma_T}.
\]

\begin{lemma}[Assumption~\eqref{asm:T}]
\label{prop:concentr_T}
    Assumption~\eqref{asm:subexp} ensures that for all $\x \ge 1$ and $n \gtrsim C_T \x$ that
    $\eps_{T}(\x) \lesssim \sigma_T\sqrt{\frac{\x}{n}}.$
\end{lemma}
\begin{proof} 
    Let $\overline{T} \eqset \frac{1}{n}\sumi T_i$. We apply Statement~\ref{stat:norm_concentration} with $\alpha = 2$ and get with probability at least $1 -  e^{-\x}$,
    \[
    \norm{\overline{T} }_{\F} \lesssim \sigma_T \sqrt{ \frac{\x}{n}} +  v_T \sqrt{\log \left( \frac{v_T}{\sigma_T}\right) } \frac{ \x }{n}.
    \]
    By substituting the condition on $n$, we get the result.
\end{proof} 

\begin{lemma}[Assumption~\eqref{asm:T_boot}]
\label{prop:concentr_T_boot}
    Set $C_w \eqdef (v_w \log v_w)^2$.
    Assumptions~\eqref{assumption:subgauss} and \eqref{asm:subexp} ensure that for all $\x, \t \ge 1$
    \[
    \hat{\eps}_T(\x,\t) \lesssim \sigma_{T} \sqrt{\frac{\x}{n}}
    \] 
    whenever $n \gtrsim C_w C_T \x (\t + \log n)$.
\end{lemma}

\begin{proof}
    First, we denote $\overline{T} \eqset T_{\hat{\mu}} - T_{\mu} = \frac{1}{n} \sumi (w_i-1) T_i.$ Note that $\overline{T}$ is centred in the bootstrap world, i.e., $\E_{w} \overline{T} = 0$. Further, $\overline{T}$ sub-Gaussian due to Assumption~\eqref{assumption:subgauss}. We apply Statement~\ref{stat:norm_concentration} and get with probability at least $1-e^{-\x}$
    \begin{equation}
    \label{eq:bound_T_F_boot}
        \norm*{\overline{T}}_{\F} \lesssim  \sqrt{\frac{1}{n}\sumi \norm{ T_i}^2_{\F} \frac{\x}{n}} + \max_{i}\norm*{\norm{(w_i - 1) T_i}_{\F}}_{\psi_{1} } z^2 \frac{\x}{n},
    \end{equation}
    with $z^2 = \log  \sqrt{ \frac{ \sumi \norm*{\norm{(w_i - 1) T_i}_{\F}}^2_{\psi_{1}}}{ \sumi \E_{w}\norm{(w_i - 1) T_i}^2_{\F} } }= \log v_w.$ Thus,
    \[
    \norm*{\overline{T}}_{\F} \lesssim  \sqrt{\frac{1}{n}\sumi \norm{ T_i}^2_{\F} \frac{\x}{n}} + v_w\log v_w\max_{i}\norm{ T_i}_{\F} \frac{\x}{n}.
    \]

    Now we apply Lemma~\ref{lemma:subexp_bounds} with $\alpha = p = 2$ and get with probability at least $1 - 2e^{-\t}$
    \begin{equation}
    \label{eq:bound_bound_bound!}
        \frac{1}{n} \sumi \norm*{T_i}_{\F}^2
        \lesssim \sigma^2_T + v^2_T  \log\left(\frac{v_T}{\sigma_T} \right)\frac{\t}{n} \lesssim \sigma^2_T.
    \end{equation}
    Moreover, $\max_i \norm*{T_i}_{\F} \lesssim v_T\sqrt{t + \log n}$. Thus, one can take
    \[
    \hat{\eps}(\x,\t) \lesssim \sigma_{T}\sqrt{\frac{\x}{n}} + v_w\log v_w v_T\sqrt{t + \log n} \frac{\x}{n} \lesssim \sigma_{T}\sqrt{\frac{\x}{n}}. \qedhere
    \]
\end{proof}

Now, we define the covariance of $T_i$ and its empirical counterpart,
\[
\Sigma \eqdef \E T_1 \otimes T_1,
\quad
\o{\Sigma}_{\mu} = \frac{1}{n} \sumi T_i \otimes T_i,
\]
with $\otimes$ being the tensor product.
And set $K_T \eqdef \norm*{\norm*{\o{\Sigma}^{-1/2} T_i }_{\F}}_{\psi_2} \le \norm{\o{\Sigma^{-1/2}}} v_T$,
$C_{G}\eqdef \left(\frac{K_T}{d}\right)^2 \log \frac{K_T}{d}.$

\begin{lemma}[Assumption~\eqref{asm:GAR}]
\label{lemma:GAR_subg}
    Under Assumption~\eqref{asm:subexp} it holds that 
    $\eps_{G} \lesssim d^3 \sqrt{\frac{C_G}{n}}.$
\end{lemma}

\begin{proof}
    The result follows from Theorem~1.1 by~\citet{bentkus_2003} applied to $X_i = \o{\Sigma}^{-1/2} T_i$ for all $i = 1, \dots, n$. Namely,
    \[
    \eps_{G} \lesssim\frac{1}{\sqrt{n}} {\E \norm*{\o{\Sigma}^{-1/2} T_i}_{\F}^3} \lesssim \frac{1}{\sqrt{n}} d^2 K_T \sqrt{\log 
    \frac{K_T}{d}},
    \]
    by Lemma~B.5 from \citep{kroshnin2024bernstein}, since $\E \norm*{\o{\Sigma}^{-1/2} T_i}_{\F}^2 = \dim \o{\Sigma} \le d^2$. 
\end{proof}

\begin{lemma}[Assumption \eqref{asm:GAR_boot}]
\label{lemma:gar_sub_boot}
    Let Assumptions~\eqref{asm:subexp} and \eqref{assumption:subgauss} be true. For sufficiently large $n$, such that
    \begin{gather}
    \label{eq:cond_on_n}
    n \gtrsim \max\left\{(\t + \log d) K_T^2\log K_T, (\t + \log d)^{3/2}\left(\frac{K_T}{d}\right)^2 \right\},
    \end{gather}
    it holds
    $\hat{\eps}_{G}(\t) \lesssim d^3 \sqrt{\frac{C_G}{n}}$.
\end{lemma}

\begin{proof}
    We denote $X_i = \frac{w_i - 1}{\sqrt{n}} \o{\Sigma}^{-1/2}_{\mu} T_i$. According to \citet{bentkus2005lyapunov}, $\hat{\eps}_{G}(\t)$ can be bounded with $(1 - \CONST e^{-\t})$-quantile of 
    \[
    \E_w \sumi \norm{X_i}_{\F}^3
    = \frac{1}{\sqrt{n}} \frac{1}{n} \sumi \E_w |w_i - 1|^{3} \norm*{\o{\Sigma_{\mu}}^{-1 / 2} T_i}^3_{\F} \lesssim  
    \frac{v_w \log v_w}{\sqrt{n}} \frac{1}{n} \sumi  \norm*{\o{\Sigma_{\mu}}^{-1 / 2} T_i}^3_{\F}.
    \]
    The last inequality is true because $\E|w_i - 1|^{3} \lesssim v_w \log v_w$.

    Now, our goal is to estimate $\lambda_{\max}(\o{I} - \o{\Sigma}^{-1/2} \o{\Sigma}_{\mu}\o{\Sigma}^{-1/2})$. 
    We will apply Bernstein inequality to random matrices $\o{I} - (\o{\Sigma}^{-1/2} T_i) \otimes (\o{\Sigma}^{-1/2} T_i)$.
    Notice that
    \begin{align*}
        \norm*{\E \left(\o{I} - (\o{\Sigma}^{-1/2} T_i) \otimes (\o{\Sigma}^{-1/2} T_i)\right)^2} 
        & = \norm*{\E \left((\o{\Sigma}^{-1/2} T_i) \otimes (\o{\Sigma}^{-1/2} T_i)\right)^2 - \o{I}}\\
        &\le \norm*{\E \left((\o{\Sigma}^{-1/2} T_i) \otimes (\o{\Sigma}^{-1/2} T_i)\right)^2}.
    \end{align*}
    For simplicity, set $Y_i = \o{\Sigma}^{-1/2} T_i$. Let $\Pi_{Y_i}$ be the orthogonal projector onto $\mathrm{span}(Y_i)$, so that $Y_i \otimes Y_i = \norm{Y_i}_\F^2 \Pi_{Y_i}$. Since $\E \norm{Y_i}_\F^2 \Pi_{Y_i} = \E Y_i \otimes Y_i = 
    \o{I}$, by Lemma B.5 in \citep{kroshnin2024bernstein} we obtain
    \begin{align*}
        \norm*{\E \left((\o{\Sigma}^{-1/2} T_i) \otimes (\o{\Sigma}^{-1/2} T_i)\right)^2} &= \norm*{\E(Y_i \otimes Y_i)^2} = \norm*{\E (\norm*{Y_i}_{\F} \Pi_{Y_i})^4}
        \\
        &\lesssim \norm*{I} K_T^2 \log \frac{K_T}{\norm{I}} =  K_T^2 \log K_T .
    \end{align*}
    
    Bernstein inequality (Theorem~1.4 from \cite{tropp2012user}) yields, with probability at least $1 -  e^{-\t}$,
    \begin{equation*}
    \label{eq:bound_on_sigma_subg}
        \lmax(\o{I} - \o{\Sigma}^{-1/2} \o{\Sigma}_{\mu}\o{\Sigma}^{-1/2}) \lesssim K_T \sqrt{\frac{\t + \log d}{n} \log K_T }
        + \frac{\t + \log d}{n}.
    \end{equation*}
    Condition \eqref{eq:cond_on_n} ensures that 
    \begin{equation}
    \label{eq:bound_lmax_nonsens}
        \lmax(\o{I} - \o{\Sigma}^{-1/2} \o{\Sigma}_{\mu}\o{\Sigma}^{-1/2}) \le \frac{1}{2},
    \end{equation} 
    thus $\norm{ \o{\Sigma}^{-1}_{\mu}} \le 2 \norm{\o{\Sigma}^{-1}}$.

    Finally, we have to estimate $\frac{1}{n} \sumi \norm*{\o{\Sigma}^{-1/2} T_i}^3_{\F}$.
    Note that $ \E \norm*{\o{\Sigma}^{-1/2} T_i}_{\F}^2 = \dim \o{\Sigma} \le d^2$.
    Applying Statement~\ref{lemma:subexp_bounds} with $p = 3$ and $\alpha = 2$, we get
    \begin{align*}
        \frac{1}{n} \sumi \norm*{\o{\Sigma}^{-1/2} T_i}^3_{\F} \lesssim d^2 K_T \sqrt{\log 
        \frac{K_T}{d}} + K_T^3\left(\left(\log \frac{K_T}{d} \right)^{3/2} + (\t + \log n)^{\frac{1}{2}} \right)\frac{\t}{n}
        \lesssim d^2 K_T \sqrt{\log \frac{K_T}{d}}.
    \end{align*}
\end{proof}

Now, we set $\sigma^2_F \eqdef  \norm*{\E \left[\dT{B}{S_1} - \o{F} \right]^2},
\quad
C_F \eqdef \frac{v^2_F}{\sigma^2_F} \log\frac{v_F}{\sigma_F}.$

\begin{lemma}[Assumption~\eqref{asm:F}]
\label{lemma:concentrationF} 
    Assumption~\eqref{asm:subexp} ensures that for all $\x > 0$ it holds that
    for sufficiently large $n$, $n \gtrsim C_F (\x + \log  d)$ that  $\eps_{F}(\x) \lesssim \norm{\o{F}^{-1}} \sigma_{F}  \sqrt{\frac{\x + \log d}{n}}.$
\end{lemma}

\begin{proof}
    We set $\o{X}_i = \dT{B}{S_i} - \E \dT{B}{S_i} = \dT{B}{S_i} - \o{F}$. By construction $\o{X}_1, \dots, \o{X}_n$ are i.i.d. Moreover, Lemma~\ref{lemma:aux_behaviour} ensures that $\norm{\o{X}_1}$ is sub-Gaussian with parameter $v_{F}$. Statement~\ref{stat:iid_concentration} ensures that with probability at least $1 - e^{-\x}$,
    \[
    \norm{\o{F} - \o{F}_{\mu}} \lesssim \sigma_{F} \sqrt{\frac{\x + \ln d}{n}} + v_F\left(\log \frac{ v_F}{\sigma_F} \right)^{1/2} \frac{\x + \ln d}{n} \lesssim \sigma_{F} \sqrt{\frac{\x + \log d}{n}}.
    \]
    Taking into account that $r(A, B) \le \norm{B^{-1}} \norm{A-B}$, we get the result.
\end{proof} 

\begin{lemma}[Assumption~\eqref{asm:F_boot}]
\label{lemma:concentration_F_flat}
    Let Assumptions~\eqref{asm:subexp} and~\eqref{assumption:subgauss} be true. For sufficiently large $n \gtrsim C_w C_F (\x + \log  d) (\t + \log n)$, it holds 
    \begin{align*}
     \hat{\eps}_{F}(\x, \t) \lesssim (\norm{\o{F}^{-1}}\sigma_F + 1)\sqrt{\frac{\x + \t + \log d}{n}}
    \end{align*}
\end{lemma}

\begin{proof} We set again $\o{X}_i = \dT{B}{S_i} - \E \dT{B}{S_i} = \dT{B}{S_i} - \o{F}$ and consider
    \begin{align*}
        \o{F}_{\hat{\mu}} - \o{F} 
        = \frac{1}{n}\sumi w_i \dT{B}{S_i} - \o{F} 
        = \frac{1}{n} \sumi (w_i-1) \o{X}_i + \o{F} \cdot \frac{1}{n} \sumi (w_i - 1) + \o{F}_{\mu} - \o{F}.
    \end{align*}
    Thus,
    \[
    r(\o{F}, \o{F}_{\hat{\mu}}) \le \norm{\o{F}^{-1}} \norm*{\frac{1}{n} \sumi (w_i-1) \o{X}_i} + \abs*{\frac{1}{n} \sumi (w_i - 1)} + r(\o{F}, \o{F}_{\mu}).
    \]

    Next, since the weights $w_i$ are sub-exponential with $\var(w)=1$, Statement~\ref{stat:iid_concentration} yields
    \[
    \left|\frac{1}{n}\sumi (w_i - 1)\right| \le \sqrt{\frac{\x}{n}} + \frac{\x}{n}v_w\log v_w
    \lesssim \sqrt{\frac{\x}{n}}.
    \]
    The last step is to bound $\frac{1}{n} \sumi (w_i-1) \o{X}_i$. We 
    apply Statement~\ref{stat:iid_concentration} and get with probability $1-e^{-\x}$
    \[
    \norm*{\frac{1}{n}\sumi (w_i-1) \o{X}_i } \lesssim \sqrt{\norm*{\frac{1}{n} \sumi \o{X}_i^2} \frac{\x + \log d}{n} } + v_w \log w \cdot \max_{i}\norm*{\o{X}_i}\frac{\x + \log d}{n} 
    \]
    
    Statement~\ref{lemma:subexp_bounds} ensures that with probability at least $1 - 2e^{-\t}$, $\max_{i}\norm*{\o{X}_i} \lesssim v_F \sqrt{\t + \log n}$.
    
    Now we set $\o{Y}_i = \o{X}_i^2$ and notice that
    \begin{gather*}
    \norm*{\E \left(\o{Y}_1 - \E\o{Y}_1 \right)^2} \le \norm*{\E\o{Y}_1^2} = \norm*{\E\o{X}_1^4} \lesssim \sigma^2_F v_F \log \frac{v_F}{\sigma_F},\\
    \norm*{\lambda_{\max}\left(\o{Y}_1 - \E\o{Y}_1 \right)_{+}}_{\psi_{1}} \le \norm*{\norm*{\o{Y}_1}}_{\psi_1} = \norm*{\norm*{\o{X}_1}^2}_{\psi_1} = v^2_F
    \end{gather*}
    Consequently, Statement~\ref{stat:iid_concentration} yields with probability at least $1-e^{-\t}$
    \[
    \norm*{
    \frac{1}{n}\sumi \o{X}_i^2} 
    \lesssim \norm*{\E \o{X}_i^2} + \sigma_F v_F \sqrt{ \frac{\t + \log d}{n} \log \frac{v_F}{\sigma_F} } 
    + v^2_F \frac{\t + \log d}{n}  \log \frac{v_F}{\sigma_F} \lesssim \sigma_F^2 .
    \]
    The last inequality holds due to the bound on $n$. Consequently,
    \begin{align*}
        \norm*{\frac{1}{n} \sumi (w_i-1) \o{X}_i } \lesssim \sigma_F \sqrt{\frac{\x + \log d}{n} } + v_w \log w \cdot v_F \sqrt{\t + \log n} \frac{\x + \log d}{n}
        \lesssim \sigma_F \sqrt{\frac{\x + \log d}{n} }.
    \end{align*}

    By Lemma~\ref{lemma:concentrationF}, $r(\o{F}, \o{F}_{\mu}) \le \eps_{F}(\t)$, with probability at least $1-e^{-\t}$.
    Combining all the bounds, we get
    \begin{align*}
        \hat{\eps}_{F}(\x,\t) \lesssim \eps_{F}(\t) + \norm{\o{F}^{-1}}\sigma_F \sqrt{\frac{\x + \log d}{n} } + \sqrt{\frac{\x}{n}} \lesssim (\norm{\o{F}^{-1}}\sigma_F + 1)\sqrt{\frac{\x + \t + \log d}{n}}.
    \end{align*}
\end{proof}

\begin{lemma}[Assumption~\eqref{asm:Sigma}]
\label{lemma:concentration_sigma_n}
    Assumption \eqref{asm:subexp} ensures for all sufficiently large $n \gtrsim \t C_T$, that
    \[
    \eps_{\Upxi}(\t) \lesssim\sigma^2_T\norm*{\o{F}^{-1}}^2 \sqrt{C_T \frac{\t + d^2}{n} }.
    \]
\end{lemma}

\begin{proof}
    Notice that
    $\norm*{\oUpxi - \oUpxi_\mu}_1 \le \norm*{\o{F}^{-1}}^2 \norm{\o{\Sigma}_{\mu} - \o{\Sigma}}_1$. 
    Further, $\E \norm*{T_1 \otimes T_1}^2_1 = \E\norm*{T_1}^{4}_{\F}$. Thus
    \[
    \E\norm{T_1\otimes T_1 - \o{\Sigma}}^2_1 
    \lesssim \E \norm*{T_1 \otimes T_1}^2_1 = \E\norm*{T_1}^{4}_{\F} 
    \lesssim \sigma^2_T v^2_T \log\frac{v_T}{\sigma_T}. 
    \]
    Moreover,
    $\norm*{\norm{T_1\otimes T_1 - \o{\Sigma}}_1 }_{\psi_1} 
    \le \norm*{\o{\Sigma}}_1 + \norm*{\norm*{T_1}^2_{\F}}_{\psi_1} \le 2\norm*{\norm*{T_1}^2_{\F}}_{\psi_1} \le 2v^2_T$. 
    Consequently, Corollary~3.5 from \citep{kroshnin2024bernstein} ensures that, with probability at least $1-e^{-\t}$,
    \[
    \norm{\o{\Sigma}_{\mu} - \o{\Sigma}}_1  \lesssim 
    \E \norm{\o{\Sigma}_{\mu} - \o{\Sigma}}_1 + \sigma_T v_T \sqrt{\frac{\t}{n} \log \frac{v_T}{\sigma_T}} + v_T^2 z \frac{\t}{n},
    \]
    where $z = \log \frac{v_T^2}{\sigma_T v_T \sqrt{\log \frac{v_T}{\sigma_T}}} \le \log \frac{v_T}{\sigma_T}$. Further,
    \begin{align*}
        \E\norm{\o{\Sigma}_{\mu} - \o{\Sigma}}_1 &\le d \E\norm{\o{\Sigma}_{\mu} - \o{\Sigma}}_2 \le d \sqrt{\E\norm{\o{\Sigma}_{\mu} - \o{\Sigma}}^2_2} = d \sqrt{\frac{1}{n} \E\norm*{T_1\otimes T_1 - \o{\Sigma} }^2_2} \\
        &\le d \sqrt{\frac{1}{n} \E\norm*{T_1\otimes T_1}^2_2} = d \sqrt{\frac{1}{n} \E \norm{T_1}^{4}_{\F}} \lesssim d \sigma_{T} v_T\sqrt{\frac{1}{n} \log\frac{v_T}{\sigma_T}}.
    \end{align*}

    Combining all the bounds, we get
    \begin{align*}
        \norm{\o{\Sigma}_{\mu} - \o{\Sigma}}_1 &\lesssim d \sigma_{T} v_T \sqrt{\frac{1}{n} \log\frac{v_T}{\sigma_T}} + \sigma_T v_T \sqrt{ \frac{\t}{n} \log\frac{v_T}{\sigma_T} } + v^2_T \log\frac{v_T}{\sigma_T} \frac{\t}{n} \\
        &\lesssim \sigma_T v_T \sqrt{ \frac{\t + d^2}{n} \log\frac{v_T}{\sigma_T} } . 
    \end{align*}
\end{proof}

\begin{lemma}[Gaussian approximation for $\bw(B, B_n)$]
\label{lemma:gar_1multiplier_boot}
    Denote 
    \[
    C_{\eps} \eqdef \kappa(B) \kappa(\o{F}) \left(c_B \sigma_T + \norm{\o{F}^{-1}} \sigma_F \right)^2, \quad
    N \eqdef \max\{ C_T, C_F \log d, C_\eps \log d\}.
    \]
    Let $n \gtrsim N$, then it holds that
    \[
    \mathcal{E} \lesssim d^3 \sqrt{\frac{C_G}{n}} + \gamma(\o{\Upxi}) \sqrt{\frac{C_\eps}{n} \log \frac{n d}{C_{\eps}}}
    \]
\end{lemma}

\begin{proof}
    Recall that the GAR bounding term is
    \begin{equation}
    \label{eq:gar_bound_again}
        \mathcal{E}  
        \lesssim \eps_{G} + \inf_{\x \in \mathscr{X}} \left\{e^{-\x} + \gamma(\oUpxi') \sqrt{\kappa(B)} \eps(\x) \right\},
        \quad
        \mathscr{X} \eqdef \left\{\x :~ \eps(\x) \le \frac{1}{6 \sqrt{\kappa(B)}} \right\}.
    \end{equation}

    We now recall that $\eps(\x) \eqdef 6 \sqrt{\kappa(\o{F})} \left( 
    c_B \eps_{T}(\x) + \eps_{\F}(\x) \right)$, with $c_B$ coming from~\eqref{def:cb}. Using Lemmata~\ref{prop:concentr_T} and~\ref{lemma:GAR_subg}, we get for any $x \ge 1$
    \[
    \eps(\x) \lesssim \sqrt{\kappa(\o{F})} \left(c_B\sigma_T \sqrt{\frac{\x}{n}} + \norm*{\o{F}^{-1}} \sigma_F \sqrt{\frac{\x + \log d}{n}}\right) \lesssim \sqrt{\frac{C_\eps}{\kappa(B) n} (\x + \log d)} .
    \]
    
    Taking $\x = \frac{1}{2} \log \frac{n}{C_\eps}$ and using assumption on $n$, we ensure that 
    \[
    \kappa(B) \eps^2(\x) \lesssim \frac{C_\eps}{n} \left(\log \frac{n}{C_\eps} + \log d\right) \lesssim 1.
    \]
    Thus, the condition $\x \in \mathscr{X}$ is satisfied. 
    Substituting $\eps_{G}$ from Lemma~\ref{lemma:GAR_subg} to \eqref{eq:gar_bound_again}, we get the result.
\end{proof}


Denote
\begin{gather*}
    \hat{C}_{\eps} \eqdef \kappa(B) \kappa(\o{F}) \left(c_B \sigma_T + \norm{\o{F}^{-1}} \sigma_F + 1\right)^2,
    \quad
    \hat{C}_{T} \eqdef  \kappa(B)\kappa^2(\o{F}), \\
    \hat{C}_G(\t) \eqdef \max\left\{(\t + \log d) K_T^2\log K_T, (\t + \log d)^{3/2}\left(\frac{K_T}{d}\right)^2 \right\}.
\end{gather*}

\begin{lemma}
\label{lemma:gar_2multiplier_boot}
    Let 
    \[
    \hat{N}(\t) \eqdef \max \left\{C_w C_T \t, C_w C_F \t \log d, \hat{C}_{\eps} (\t + \log d), \hat{C}_G(\t), \hat{C}_{T}\t
    \right\},
    \]
    then for $n\gtrsim \hat{N}(\t)$ with probability $1-\CONST e^{-\t}$
    \[
    \hat{\mathcal{E}}(\t) \lesssim \hat{\eps}_{G}(\t) + \gamma(\o{\Upxi}'_\mu) \left(1 + \sqrt{\frac{\tr(\o{\Upxi}')}{\tr( \o{\Upxi}'_\mu)}}\right) \sqrt{\frac{\hat{C}_{\eps}}{n} \left(\t + \log \frac{n d}{\hat{C}_{\eps}}\right)} .
    \]
\end{lemma}

\begin{proof}
    To get the bound on the random variable $\hat{\mathcal{E}}(\t)$, we note that Lemma~\ref{lemma:gar_boot_BW} ensures with probability $1-\CONST e^{-\t}$ 
    \begin{gather}
        \hat{\mathcal{E}}(\t) \lesssim \hat{\eps}_{G}(\t) + \inf_{\x \in \hat{\mathscr{X}}(\t)} \left\{ e^{-\x} + \gamma(\o{\Upxi}'_\mu) \sqrt{\kappa(B)} \left(\eps(\t) + \hat{\eps}(\x, \t)\right)
        \left( \tfrac{\norm{\A} \norm{ \o{F}^{-1}}}{\sqrt{\tr(\o{\Upxi}'_\mu)}} \eps_{T}(\t) + 1 \right)\right\},
    \end{gather}
    where
    $\hat{\mathscr{X}}(\t) \eqdef \left\{\x:~ \hat{\eps}(\x, \t) \le \frac{1}{12 \sqrt{ \kappa(B)}}\right\}$. 
    First, we use Lemmata~\ref{prop:concentr_T_boot} and~\ref{lemma:concentration_F_flat} and get
    \begin{align*}
       \hat{\eps}(\x, \t) & \eqdef 6 \sqrt{\kappa(\o{F})} \left( c_B\hat{\eps}_{T}(\x, \t) + \hat{\eps}_{F}(\x, \t)  \right) \\
       &\lesssim \sqrt{\kappa(\o{F})}\left( c_B\sigma_{T}\sqrt{\frac{\x}{n}} + (\sigma_F \norm{\o{F}^{-1}} + 1) \sqrt{\frac{\x + \t + \log d}{n}}\right) \\
       &\lesssim \sqrt{\frac{\hat{C}_{\eps}}{\kappa(B) n} (\x + \t + \log d)} . 
    \end{align*}

    Condition on $n$ yields
    \[
    \frac{\norm{\A} \norm{\o{F}^{-1}}}{\sqrt{\tr(\o{\Upxi}'_\mu)}} \eps_{T}(\t) 
    \lesssim \frac{\norm{\A} \norm{\o{F}^{-1}}}{\sqrt{\tr(\o{\Upxi}'_\mu)}} \sigma_T \sqrt{\frac{\t}{n}}
    \lesssim \frac{1}{\norm{\A^{-1}} \norm{\o{F}}} \sqrt{\frac{\tr \o{\Sigma}}{\tr(\o{\Upxi}'_\mu)}}
    \lesssim \sqrt{\frac{\tr(\o{\Upxi}')}{\tr(\o{\Upxi}'_\mu)}} .
    \]

    Next, according to the proof of Lemma~\ref{lemma:gar_1multiplier_boot}, $\eps(\t) \lesssim \sqrt{\frac{C_{\eps}}{\kappa(B) n} \t} \le \sqrt{\frac{\hat{C}_{\eps}}{\kappa(B) n} \t}$.
    
    Taking $\x = \frac{1}{2} \log \frac{n}{\hat{C}_\eps}$, we obtain that $\kappa(B) (\eps(\t) + \hat{\eps}(\x; \t))^2 \lesssim 
    \frac{\hat{C}_{\eps}}{n} (\x + \t + \log d) \lesssim 1,$ 
    hence $\x \in \hat{\mathscr{X}}(\t)$, and
    \[
    \hat{\mathcal{E}}(\t) \lesssim \hat{\eps}_{G}(\t) + \gamma(\o{\Upxi}'_\mu) \left(1 + \sqrt{\frac{\tr(\o{\Upxi}')}{\tr( \o{\Upxi}'_\mu)}}\right) \sqrt{\frac{\hat{C}_{\eps}}{n} \left(\t + \log \frac{n d}{\hat{C}_{\eps}}\right)} .
    \]
\end{proof}

Before proving the theorem, we collect some definitions used throughout the text below for completeness.
The constants from lemmata that ensure GAR,
\begin{gather}
    C_w \eqdef (v_w \log v_w)^2,
    \quad
    C_T \eqdef \frac{v^2_T}{\sigma^2_T} \log\frac{v_T}{\sigma_T},
    \\
    \hat{C}_{T} \eqdef \kappa(B)\kappa^2(\o{F}),
    \quad
    C_F \eqdef \frac{v^2_F}{\sigma^2_F} \log\frac{v_F}{\sigma_F}, \label{def:ct}\\
    K_T \eqdef \norm*{\norm*{\o{\Sigma}^{-1/2} T_i }_{\F}}_{\psi_2} \le \norm{\o{\Sigma^{-1/2}}} v_T,
    \quad
    C_{G}\eqdef \left(\frac{K_T}{d}\right)^2 \log \frac{K_T}{d}.\label{def:cg_hat}
\end{gather}
Moreover, the constants coming from Lemma~\ref{lemma:gar_1multiplier_boot} and Lemma~\ref{lemma:gar_2multiplier_boot}

\begin{gather}
    C_{\eps} = \kappa(B) \kappa(\o{F}) \left(c_B \sigma_T + \norm{\o{F}^{-1}} \sigma_F \right)^2,\\ 
    \hat{C}_{\eps} \eqdef \kappa(B) \kappa(\o{F}) \left(c_B \sigma_T + \norm{\o{F}^{-1}} \sigma_F + 1\right)^2, \label{def:ce_hat}\\
    \hat{C}_G(\t) \eqdef \max\left\{(\t + \log d) K_T^2\log K_T, (\t + \log d)^{3/2}\left(\frac{K_T}{d}\right)^2 \right\} \nonumber.
\end{gather}

In the following, we assume that
\begin{gather}
    n \gtrsim \max\{N, \hat{N}(\t), \t C_T\}, \quad
    N \eqdef \max\{ C_T, C_F \log d, C_\eps \log d\},\label{def:size_n}\\
    \hat{N}(\t) \eqdef \max \left\{C_w C_T \t, C_w C_F \t \log d, \hat{C}_{\eps} (\t + \log d), \hat{C}_G(\t), \hat{C}_{T}\t
    \right\}.\nonumber
\end{gather}

\begin{proof}[Proof of Theorem~\ref{thm:bootstrap_bw}]
    If $W$ is such that $\P_{w}\{w =0\} = 0$, the proof is trivial and reduces to validation of all assumptions in Theorem~\ref{theorem:boot_BW}.
    
    Now we consider the weight generating law $W$, such that $\P_{w}\{w =0\} = p_0$. Let an auxiliary measure $\tilde{\mu}$ be
    \[
    \tilde{\mu} = \sumi w_i\delta_{S_i},~~\text{s.t.}~\sumi w_i \neq 0,
    \]
    and set, w.l.o.g., $\mathscr{B}(0) \eqdef B_0$ with $B_0 \in \H_{++}(d)$ being some fixed matrix.
    
    We aim to show that
    \begin{equation}
    \label{eq:aux_bound1}
        \left|\P\left\{\bw(B_{\tilde{\mu}}, B_{\mu} ) \le z | \mu   \right\} - \P\left\{  \bw(B_{\hat{\mu}}, B_{\mu} )  \le z | \mu \right\}   \right| \le p^n_0.
    \end{equation}
    
    We will use the following facts:
    \begin{align*}
        &\P\left\{A|B \right\} - \P\left\{A \right\} = \frac{\P\left\{A \cap B \right\}}{\P\left\{B \right\}} - \P\left\{A \right\} \\
        &\qquad\le \P\left\{A \right\}
        +\left(\frac{1}{\P\left\{ B\right\}} - 1\right) \P\left\{ B\right\}  - \P\left\{ A\right\} \le 1 - \P\left\{ B\right\},\\
        &\P\left\{A \right\} - \P\left\{A|B \right\} \le \P\left\{A \right\} - \P\left\{A \cap B \right\} \le 1 - \P\left\{B \right\} .
    \end{align*}
    
    Thus, for a fixed set $S_1, \dots, S_n$,
    \[
    \left| \P_{w }\left\{\bw(B_{\hat{\mu}}, B_{\mu}) \le z \biggm| \sumi w_i \neq 0 \right\} - \P_{w }\left\{\bw(B_{\tilde{\mu}}, B_{\mu}) \le z  \right\} \right| \le \P_{w }\left\{\sumi w_i= 0\right\} = p^n_0.
    \]
    
    Now, we notice that the condition $\sumi w_i = 0$ is equivalent to $\hat{\mu} = 0$. Thus, \eqref{eq:aux_bound1} follows from
    \begin{align*}
        \Bigl|\P\left\{\bw(B_{\tilde{\mu}}, B_{\mu} ) \le z \middle| \mu\right\} &- \P\left\{\bw(B_{\hat{\mu}}, B_{\mu}) \le z \middle| \mu \right\} \Bigr| \\
        &= \left|\P\left\{\bw(B_{\tilde{\mu}}, B_{\mu} ) \le z \middle| \mu \right\} - \P\left\{  \bw(B_{\tilde{\mu}}, B_{\mu}) \le z \middle| \mu, \hat{\mu} \neq 0 \right\} \right|\\
        &\le \P\left\{ \hat{\mu} = 0 \middle| \mu \right\} = \P_{w}\left\{\sumi w_i = 0\right\} = p^n_0.
    \end{align*}
    
    Further, Lemma~\ref{lemma:gar_boot_BW}, being applied to $\tilde{\mu}$ (instead of $\hat{\mu}$) together with the above bound, yields for all $z > 0$
    \[
    \abs*{
    \P\left\{\bw(B_{\hat{\mu}}, B_{\mu}) \le z ~|~\mu\right\} - 
    \P\left\{\norm{\A Z_{\mu}}_{\F} \le z ~|~\mu\right\}
    }
    \le \hat{\mathcal{E}}(\t) + p^n_0.
    \]
    Thus, the resulting bound is written as
    \[
    \sup_{z\ge 0} \abs*{
     \P\left\{\bw(B_{\mu},B) \le z\right\} 
     - \P\left\{ \bw(B_{\hat{\mu}},B_{\mu}) \le z ~|~\mu\right\}} \leq \Gamma(\t) + p^n_0.
    \]
    
    Finally, to get the asymptotic bound on $\Gamma(\t) + p^n_0$ for large $n$, we summarize all auxiliary results from this section. 
    
    To get the second result, we recall Theorem~\ref{theorem:boot_BW} and notice that
    \[
    \Gamma(\t) \lesssim \varkappa( \oUpxi')\norm{\A}^2 \eps_{\Upxi}(\t) + \mathcal{E} + \hat{\mathcal{E}}(\t).
    \]
    First, we recall Lemma~\ref{lemma:gar_2multiplier_boot},
    \[
    \hat{\mathcal{E}}(\t) \lesssim \hat{\eps}_{G}(\t) + \gamma(\o{\Upxi}'_\mu) \left(1 + \sqrt{\frac{\tr(\o{\Upxi}')}{\tr( \o{\Upxi}'_\mu)}}\right) \sqrt{\frac{\hat{C}_{\eps}}{n} \left(\t + \log \frac{n d}{\hat{C}_{\eps}}\right)}
    \]
    Assumption on $n$ ensures $\gamma(\o{\Upxi}'_\mu)\lesssim \gamma(\o{\Upxi}')$, $\tr (\o{\Upxi}'_\mu) \lesssim \tr \gamma(\o{\Upxi}')$ (see \ref{eq:yet_another_name}). Using Lemmata~\ref{lemma:gar_1multiplier_boot}, \ref{lemma:gar_2multiplier_boot}, \ref{lemma:concentration_sigma_n}, and the fact that by definition $\hat{C}_\eps > C_\eps$, we get
    \[
    \Gamma(\t) \lesssim d^3\sqrt{\frac{C_G}{n}} + \gamma(\oUpxi')\sqrt{\frac{\hat{C}_\eps}{n}(\t + \log \frac{nd}{\hat{C}_\eps})} + \varkappa(\oUpxi')\norm*{\A}^2\norm{\o{F}^{-1}}^2\sigma^2_T\sqrt{\frac{C_T}{n} (\t + d^2)}.
    \]
    Finally, 
    $\gamma(\oUpxi') = \varkappa(\oUpxi')\tr\oUpxi' \le \varkappa(\oUpxi') \norm*{\A}^2\norm{\o{F}^{-1}}^2 \tr \o{\Sigma} = \varkappa(\oUpxi') \norm*{\A}^2\norm{\o{F}^{-1}}^2 \sigma^2_T.$  Combining the bounds, we get the result.
\end{proof}

\section{Computational aspects}
\label{appx:computation}

\begin{lemma}\label{lemma:bound_on_lambda}
    Let $Q_0 \in \H_{++}(d)$, $Q_1 = f_\mu(Q_0)$. Denote
    \[
    A_\mu \eqdef \left(\int\limits_{\H_{++}(d)} \lmax^{1/2}(S) \dd \mu(S) \right)^2,
    \quad
    a_\mu \eqdef \left(\int\limits_{\H_{++}(d)} \lmin^{1/2}(S) \dd \mu(S) \right)^2,
    \quad
    \kappa_{\mu} \eqdef \frac{A_\mu}{a_\mu}.
    \]
    Then
    \begin{align*}
        \lmax(Q_1) \le A_{\mu} , \quad
        \lmin(Q_{1}) \ge \frac{a_{\mu}}{\sqrt{\kappa(Q_0)}} ,
        \quad \kappa(Q_1) \le \kappa_\mu \sqrt{\kappa(Q_0)} .
    \end{align*}
\end{lemma}

\begin{proof}
    First, recall that the original algorithm is written as
    \begin{equation}
        f_\mu(Q) = Q^{-1/2} \left [\int\limits_{\H_{++}(d)} \left(Q^{1/2} S  Q^{1/2}\right)^{1/2} \dd \mu(S) \right]^2 Q^{-1/2}, \quad Q \in \H_{++}(d),
    \end{equation}
    see \cite{alvarez2016fixed}.
    Denote $R(Q) \eqdef \left(Q^{1/2} S  Q^{1/2}\right)^{1/2}$. Jensen's inequality yields
    \begin{align*}
        \lmax(Q_1) &= \norm*{Q^{-1/2}_0 \int\limits_{\H_{++}(d)} R(Q_0) \dd \mu(S) }^2 \le 
        \left(\int\limits_{\H_{++}(d)} \norm*{Q^{-1/2}_0  R(Q_0)} \dd \mu(S)\right)^2 = A_k,
    \end{align*}
    since 
    \[
    \norm*{Q^{-1/2}_0  R(Q_0)}^2
    = \lmax\left(Q^{-1/2}_0  R^2(Q_0) Q^{-1/2}_0\right) = 
    \lmax(S) .
    \]

    Next, we notice that by construction
    \begin{align*}
        Q_1 &\succcurlyeq \lmin\left(\int\limits_{\H_{++}(d)} R(Q_0) \dd \mu(S) \right) Q^{-1/2}_0 \left[ \int\limits_{\H_{++}(d)} R(Q_0) \dd \mu(S) \right]Q^{-1/2}_0\\
        &\succcurlyeq \lmin\left(\int\limits_{\H_{++}(d)} R(Q_0) \dd \mu(S) \right) Q^{-1/2}_0 \left[\int\limits_{\H_{++}(d)}Q^{1/2}_0\lmin^{1/2}(S) \dd \mu(S)\right]  Q^{-1/2}_0\\
        &\succcurlyeq \lmin^{1/2}(Q_0)a_{\mu} Q^{-1/2}_0.
    \end{align*}
    Consequently,
    \[
    \lmin(Q_1) \ge \frac{\lmin^{1/2}(Q_0)}  {\lmax^{1/2}(Q_0)} a_\mu = \frac{1}{\sqrt{\kappa(Q_0)}}a_\mu. 
    \]
    Finally, this ensures
    $\kappa(Q_1) \le \sqrt{\kappa(Q_0)} \kappa_\mu$.
\end{proof}

\begin{proof}[Proof of Theorem~\ref{thm:complexity}]
    Let $Q_{t} = tQ + (1-t)B_{\mu}$ with $t\in [0, 1]$, where $B_{\mu}$ is the barycenter w.r.t.\ $\mu$. We set
    \begin{equation*}
        \o{D}_{\mu}[Q] \eqdef - \underset{\H_{++}(d)}{\int} \left[ \int\limits_{0}^{1} \dT{Q_t}{S}\dd t\right] \dd \mu(S).
    \end{equation*}
    Lemma~A.6 by \cite{kroshnin2019statistical} ensures that 
    \begin{equation}
    \label{eq:aux_for_frob_norm}
        \V_\mu(Q_k) -  \V(B_\mu) \le  
        \left\langle \o{D}_\mu[Q_k]\left(Q_k - B_\mu\right), Q_k - B_\mu
        \right\rangle.
    \end{equation}
    The Taylor formula for $\mathscr{B}$ (see Theorem~2.2 by~\citet{kroshnin2019statistical}) yields
    $
    B_\mu - Q_k = \o{D}^{-1}_{\mu}[Q_k] T_\mu[Q_k].
    $
    Thus, 
    \begin{equation}
    \label{eq:bound_V_k}
        \V_\mu(Q_k) -  \V(B_\mu) \le 
        \left\langle T_\mu[Q_k], \o{D}^{-1}_{\mu}[Q_k] T_\mu[Q_k]
        \right\rangle.
    \end{equation}
    Denote
    \begin{gather*}
        \o{d}_{\mu}(X) \eqdef Q^{1/2}_k \o{D}^{-1}_{\mu}[Q_k] \left(Q^{1/2}_k  X Q^{1/2}_k \right) Q^{1/2}_k, \\
        \o{D}^{-1}_\mu(Y) = Q^{1/2}_k \o{d}^{-1}_{\mu}(Q^{1/2}_kY Q^{1/2}_k)Q^{1/2}_k 
    \end{gather*}
    Next, for the sake of clarity, we will write $\o{D}_\mu$ instead of $\o{D}_\mu[Q_k]$ and $T_\mu$ instead of $T_\mu[Q_k]$, 
    \begin{align*}
        \left \langle
        \o{D}^{-1}_\mu T_\mu, T_\mu\right\rangle
        &= \left \langle
        Q^{1/2}_k T_\mu Q^{1/2}_k,
        \o{d}^{-1}_\mu\left(Q^{1/2}_k T_\mu Q^{1/2}_k \right)
        \right\rangle \\
        &\le \frac{1}{\lmin(\o{d}_{\mu})}\norm*{Q^{1/2}_k T_\mu Q^{1/2}_k}_\F
        \le \frac{\lmax(Q_k)}{\lmin(\o{d}_{\mu})} \norm*{Q^{1/2}_k T_\mu}^2_\F.
    \end{align*}
    Set $Q'_k \eqdef B^{1/2}_\mu Q_k B^{1/2}_\mu$. Next, we notice that (see Corollary~A.2 in \cite{kroshnin2019statistical})
    \[
    \lmin(\o{d}_\mu) \ge \frac{2 \lmin^{1/2}(Q_k) a^{1/2}_\mu}{\lmax^{1/2}(Q'_k) + \lmax(Q'_k)}.
    \]
    Next, we notice that Lemma~\ref{lemma:bound_on_lambda} ensures $\lmax(Q'_k) \le \frac{\lmax(B_\mu)}{\lmin(Q_k)} \le \frac{A_{\mu}}{a_\mu / \kappa_\mu} = \kappa_\mu.$ Thus, 
    \[
    \lmin(\o{d}_\mu) \ge \frac{\lmin^{1/2}(Q_k) a^{1/2}_\mu}{\kappa^2_\mu} \ge \frac{a_\mu}{\kappa^{5/2}_\mu}.
    \]
    Consequently,
    $\left\langle \o{D}^{-1}_\mu T_\mu, T_\mu \right\rangle 
    \le A_\mu \frac{\kappa^{5/2}_\mu}{a_\mu} 
    \left\langle T_\mu, Q_k T_\mu \right\rangle
    = \kappa^{7/2}_\mu \left\langle T_\mu, Q_k T_\mu \right\rangle$.
    Thus, \eqref{eq:bound_V_k} can be rewritten as
    \[
    \V_\mu(Q_k) -  \V(B_\mu) \le \kappa^{7/2}_\mu\left  \langle
    T_\mu, Q_k T_\mu
    \right \rangle.
    \]
    Next, we note that (see, e.g., Proposition~3.3 by \cite{alvarez2016fixed} )
    \begin{equation}\label{eq:bound_V_k+1}
        \V_\mu(Q_{k+1}) - \V_\mu(B_\mu) \le \V_\mu(Q_{k}) - \V_\mu(B_\mu) - \bw^2(Q_{k+1}, Q_k).
    \end{equation}
    It is easy to see (e.g., eq.~(23) in \cite{alvarez2016fixed}) that
    \begin{equation*}
        \bw^2(Q_{k+1}, Q_{k}) = \tr(Q_k) + \tr(Q_{k+1}) - 2\tr(Q_k \E_{S\sim \mu} T^{S}_{Q_k}) 
        = \bigl \langle
         T_{\mu}, Q_k T_{\mu}
        \bigr \rangle.
    \end{equation*}
    Thus, 
    \begin{align*}
        \V_\mu( Q_{k+1}) - \V_\mu(B_\mu)& \le (1-\kappa^{-7/2})\left(\V_\mu(Q_{k}) - \V_\mu(B_\mu)   \right) \\
        &\le (1-\kappa^{-7/2})^k_\mu\left(\V_\mu(Q_{0}) - \V_\mu(B_\mu)   \right).
    \end{align*}
    Recall, that $\o{F}_{\mu}\eqdef -\underset{\H_{++}(d)}{\int} \dT{B_\mu}{S}\dd \mu(S)$. 
    Lemma~A.6 from \cite{kroshnin2019statistical} ensures that
    \begin{align*}
        \V_{\mu}(Q_k) - \V_{\mu}(B_{\mu}) &\ge \frac{1}{\left(1 + \lmax^{1/2}(Q'_k) \right)^2} \left\langle 
        \o{F}_{\mu}\left(Q_k - B_\mu\right), Q_k - B_\mu
        \right\rangle  \\
        &\ge \frac{\lmin(\o{F}_\mu)}{\left(1 + \lmax^{1/2}(Q'_k) \right)^2} \norm*{Q_k - B_\mu}^2_{\F}.
    \end{align*}
    Thus, 
    \begin{align*}
        \norm*{Q_k - B_\mu}_{\F} &\le  \frac{\left(1 + \lmax^{1/2}(Q'_k) \right)}{\lmin^{1/2}(\o{F}_\mu)} \left(\V_{\mu}(Q_k) - \V_{\mu}(B_\mu)\right) \\ 
        &\le \frac{2\kappa^{1/2}_\mu }{\lmin(\o{F}_{\mu})} (1-\kappa^{-7/2}_\mu)^{k/2}\left(\V_\mu(Q_{0}) - \V_\mu(B_\mu) \right)^{1/2}. 
    \end{align*}
    Using the triangle inequality so that to bound $\norm*{Q_{k+1} - Q_{k}}_\F$, we get that for the given precision $\eps>0$ it is enough to make $k \ge N$ steps with
    \[
    N = 2\kappa^{7/2}_\mu \ln
    \left(
    \frac{1}{\eps}\cdot \frac{2\kappa^{1/2}_\mu}{\lmin(\o{F}_{\mu})}\left(\V_\mu(Q_{0}) - \V_\mu(B_\mu) \right)^{1/2}  
    \right)
    \]
\end{proof}

\end{appendix}



\end{document}

%% file: mydef.tex
\renewcommand{\o}[1]{\bm{#1}}
\newcommand*{\sumi}{\sum_{i = 1}^n}
\newcommand{\cc}[1]{\mathscr{#1}}
\def\ND{\cc{N}}

\newcommand{\eqdef}{\stackrel{\operatorname{def}}{=}}
\newcommand{\eqset}{\coloneqq}

\renewcommand{\Gamma}{\varGamma}
\renewcommand{\Pi}{\varPi}
\renewcommand{\Sigma}{\varSigma}
\renewcommand{\Delta}{\varDelta}
\renewcommand{\Lambda}{\varLambda}
\renewcommand{\Psi}{\varPsi}
\renewcommand{\Phi}{\varPhi}
\renewcommand{\Theta}{\varTheta}
\renewcommand{\Omega}{\varOmega}
\renewcommand{\Xi}{\varXi}
\renewcommand{\Upsilon}{\varUpsilon}

\DeclareMathOperator*{\aargmin}{Argmin}
\DeclareMathOperator*{\argmin}{argmin}

\let\norm\relax
\DeclarePairedDelimiter{\norm}{\lVert}{\rVert}
\let\abs\relax
\DeclarePairedDelimiter{\abs}{\lvert}{\rvert}

\newcommand{\dd}{\,d}
\newcommand*{\R}{\mathbb{R}}

\renewcommand*{\H}{\mathbb{H}}

\newcommand*{\A}{\mathbb{A}}

\DeclareMathOperator{\diag}{diag}
\DeclareMathOperator{\tr}{tr}
\DeclareMathOperator{\range}{range}

\renewcommand*{\P}{\mathbb{P}}
\DeclareMathOperator{\E}{\mathbb{E}}
\DeclareMathOperator{\var}{Var}

\newcommand{\V}{\mathcal{V}}

\def\CONST{\mathtt{C}}

\newcommand*{\eps}{\varepsilon}
\newcommand*{\lmax}{\lambda_{\max}}
\newcommand*{\lmin}{\lambda_{\min}}
\newcommand*{\diff}{\,d}

\newcommand*{\dT}[2]{\o{d T}_{#1}^{#2}}

\newcommand*{\Id}{\o{Id}}

\renewcommand{\tilde}[1]{\widetilde{#1}}

\newcommand{\Tmw}{T_{\mu}}
\renewcommand{\Id}{\o{I}}

\newcommand{\F}{\mathrm{F}}

\renewcommand{\A}{\o{A}}

\newcommand{\oUpxi}{\o{\Upxi}}

\newcommand{\x}{\mathrm{x}}
\newcommand{\y}{\mathrm{y}}

\renewcommand{\t}{\mathrm{t}}

\newcommand{\bw}{\mathscr{W}}